\documentclass[12pt,twoside,reqno]{amsart}
\usepackage{amsmath}
\usepackage{amsfonts}
\usepackage{amssymb}
\usepackage{color}
\usepackage{mathrsfs}
\usepackage{cite}
\usepackage{geometry}
\usepackage{marginnote}
\usepackage{todonotes}
\allowdisplaybreaks
\newtheorem{theorem}{Theorem}[section]
\newtheorem{corollary}[theorem]{Corollary}
\newtheorem{lemma}[theorem]{Lemma}
\newtheorem{proposition}[theorem]{Proposition}
\newtheorem{definition}[theorem]{Definition}
\newtheorem{remark}[theorem]{Remark}
\numberwithin{equation}{section}
\begin{document}
\title{Existence and non-existence for\\ the collision-induced breakage equation} 
\thanks{Partially supported by the Indo-French Centre for Applied Mathematics (IFCAM) within the project \textsl{Collision-induced fragmentation and coagulation: dynamics and numerics}}

\author{Ankik Kumar Giri}
\address{Department of Mathematics, Indian Institute of Technology Roorkee\\ Roorkee--247667, Uttarakhand, India}
\email{ankik.giri@ma.iitr.ac.in/ankik.math@gmail.com}

\author{Philippe Lauren\c{c}ot}
\address{Institut de Math\'ematiques de Toulouse, UMR~5219, Universit\'e de Toulouse, CNRS \\ F--31062 Toulouse Cedex 9, France}
\email{laurenco@math.univ-toulouse.fr}


\date{\today}

\begin{abstract}
A mathematical model for collision-induced breakage is considered. Existence of weak solutions to the continuous nonlinear collision-induced breakage equation is shown for a large class of unbounded collision kernels and daughter distribution functions, assuming the collision kernel $K$ to be given by $K(x,y)= x^{\alpha} y^{\beta} + x^{\beta} y^{\alpha}$ with $\alpha \le \beta \le 1$. When $\alpha + \beta \in [1,2]$, it is shown that there exists at least one weak mass-conserving solution for all times. In contrast, when $\alpha + \beta \in [0,1)$ and $\alpha \ge 0$, global mass-conserving weak solutions do not exist, though such solutions are constructed on a finite time interval depending on the initial condition. The question of uniqueness is also considered. Finally, for $\alpha <0$ and a specific daughter distribution function, the non-existence of mass-conserving solutions is also established.
\end{abstract}

\maketitle

%
%
\pagestyle{myheadings}
\markboth{\sc{A.K. Giri \& Ph. Lauren\c cot}}{\sc{The collision-induced breakage equation}}

\section{Introduction}\label{sec1}

The collision-induced breakage equation, also referred to as the nonlinear fragmentation equation, describes the dynamics of a large number of particles breaking apart as a result of binary collisions and arises in the modeling of cloud drops formation \cite{LiGi1976, Sriv78} and planet formation \cite{BKBSHSS15, Safr72}. Specifically, denoting the size distribution function of particles of size $x\in (0,\infty)$ at time $t\ge 0$ by $f=f(t,x)\ge 0$, the evolution of $f$ is governed by the nonlinear nonlocal equation 
\begin{subequations}\label{pa1}
\begin{align}
\partial_t f(t,x) & = \mathcal{B} f(t,x) \,, \qquad (t,x)\in (0,\infty)^2\,, \label{pa1a} \\
f(0,x) & = f^{in}(x)\ge 0 \,, \qquad x\in (0,\infty)\,, \label{pa1b}
\end{align}
where
\begin{equation}
\begin{split}
\mathcal{B} f(x) & := \frac{1}{2} \int_x^\infty \int_0^y b(x,y-z,z) K(y-z,z) f(y-z) f(z)\ dzdy \\
& \qquad - \int_0^\infty K(x,y) f(x) f(y)\ dy \,,
\end{split} \label{pa1c}
\end{equation}
\end{subequations}
for $x\in (0,\infty)$. In \eqref{pa1c}, collisions are monitored by the collision kernel $K$, while the outcome of collisional fragmentation is specified by the daughter distribution function $b$. In fact, $K(x,y)=K(y,x)\ge 0$ measures the likeliness of pairwise collisions between particles with respective sizes $x\in (0,\infty)$ and $y\in (0,\infty)$, while $b(z,x,y)$ provides the fraction of particles of size $z\in (0,x+y)$ produced after the collision between particles with respective sizes $x\in (0,\infty)$ and $y\in (0,\infty)$. Equation~\eqref{pa1} features a gain term accounting for the formation of particles of size $x$ resulting from the collision between particles with respective sizes $y-z$ and $z$ with $y>x$ and $z\in (0,y)$ and a loss term describing the disappearance of particles of size $x$ as they collide with other particles with arbitrary size. Since there is neither creation nor loss of matter during breakup, we assume that, for $(x,y)\in (0,\infty)^2$,
\begin{equation}
\int_0^{x+y} z b(z,x,y)\ dz = x+ y \;\;\text{ and }\;\; b(z,x,y) = 0 \;\text{ for }\; z>x+y\,. \label{pa2}
\end{equation}
A (formal) consequence of \eqref{pa2} on the dynamics of \eqref{pa1} is that solutions to \eqref{pa1} are expected to satisfy the mass conservation 
\begin{equation}
\int_0^\infty x f(t,x)\ dx  = \int_0^\infty x f^{in}(x)\ dx\,, \qquad t\ge 0\,, \label{zp1}
\end{equation}
at least when the right hand side of \eqref{zp1} is finite. Another physical property embedded in \eqref{pa2} is that the collision of two particles with respective sizes $x\in (0,\infty)$ and $y\in (0,\infty)$ does not produce fragments with a size exceeding $x+y$. Still, it is worth pointing out here that, in general, mass transfer is allowed during collision-induced breakage, in the sense that the collision between two particles with respective sizes $x$ and $y$ may produce a particle of size bigger than $\max\{x,y\}$, after transfer of matter from the smallest particle to the largest one (for instance, $\{x\} + \{y\} \longrightarrow \{x/2\} + \{y+x/2\}$). However, mass transfer is excluded if there is a non-negative function $\bar{b}$ such that, for $(x,y)\in (0,\infty)^2$, 
\begin{subequations}\label{pa3}
\begin{align}
& b(z,x,y) = \bar{b}(z,x,y) \mathbf{1}_{(0,x)}(z) + \bar{b}(z,y,x) \mathbf{1}_{(0,y)}(z)\,, \qquad z\in (0,\infty)\,, \label{pa3a} \\
& \int_0^x z \bar{b}(z,x,y)\ dz = x \;\;\text{ and }\;\; \bar{b}(z,x,y) = 0 \;\text{ for }\; z>x\,, \label{pa3b}
\end{align}
\end{subequations}
see \cite{ChRe88,ChRe90}. Clearly, \eqref{pa3} implies \eqref{pa2}. A typical example is given by
\begin{equation}
	\bar{b}_\nu(z,x,y) := (\nu+2) z^{\nu} \mathbf{1}_{(0,x)}(z) x^{-\nu-1}\,, \qquad (x,y,z)\in (0,\infty)^3\,, \qquad \nu\in (-2,0]\,, \label{zp4}
\end{equation}
see \cite{ErPa07}.

Unlike its linear counterpart, the linear or spontaneous fragmentation equation, which has received considerable attention since the pioneering works of Filippov \cite{Fili61}, Kapur \cite{Kapu72a}, and McGrady \& Ziff \cite{McZi87, ZiMc85}, see \cite{BLL19, Bert06} and the references therein, fewer works are devoted to the collision-induced breakage equation \eqref{pa1} and are found in the physics literature, mostly dealing with the scaling behaviour and shattering transition \cite{ChRe88, ChRe90, ErPa07, KoKa00, KoKa06, KrBN03, VVF06}.  The purpose of this work is then to investigate the basic issues of existence, uniqueness, and non-existence of weak solutions to \eqref{pa1}, as well as that of mass conservation for these solutions. In fact, we shall show that the existence of mass-conserving weak solutions to \eqref{pa1} and their lifetime strongly depend on the growth of the collision kernel $K$ for large and small sizes. Roughly speaking, mass-conserving weak solutions exist globally when $K$ grows at least linearly for large sizes,  Theorem~\ref{thma2}. Sublinear growth of $K$ for large sizes impedes global existence but local existence of mass-conserving weak solutions can still be shown, see Theorems~\ref{thma2} and~\ref{thma3}. Finally, unboundedness of $K$ for small sizes prevents the existence of non-zero mass-conserving weak solutions, even on short time intervals, see Theorem~\ref{thma6}. Such a behaviour is reported in \cite{ErPa07} for a particular choice of $(K,b)$, which allows one to map \eqref{pa1} to a linear fragmentation equation with a different time scale and deduce the dynamics of the former from that of the latter, which is already well-documented, as already mentioned. We here extend the validity of these results to a broader class of collision kernels $K$ and daughter distribution functions $b$ and provide analytical proofs, without having recourse to a transformation to another equation.

Let us now describe more precisely the collision kernels and daughter distribution functions to be dealt with in the sequel, as well as the functional analytic framework and the notion of weak  solutions to \eqref{pa1} needed for our analysis. First, besides the conservation of matter \eqref{pa2}, we shall assume that the number of particles resulting from the collisional breakage of two particles is bounded whatever the sizes of the incoming particles; that is, there is $\beta_0 > 2$ such that
\begin{equation}
\int_0^{x+y} b(z,x,y)\ dz \le \beta_0\,, \qquad (x,y)\in (0,\infty)^2\,. \label{pa4}
\end{equation}

\begin{remark}\label{rema0}
When $b$ satisfies \eqref{pa3} and \eqref{pa4}, the assumed lower bound $\beta_0>2$ is actually a consequence of \eqref{pa3b}. Indeed, for $(x,y)\in (0,\infty)^2$, 
\begin{align*}
 \beta_0 \ge  \int_0^{x+y} b(z,x,y)\ dz & = \int_0^x \bar{b}(z,x,y)\ dz + \int_0^y \bar{b}(z,y,x)\ dz \\
& > \frac{1}{x} \int_0^x z \bar{b}(z,x,y)\ dz + \frac{1}{y} \int_0^y z \bar{b}(z,y,x)\ dz =2\,.
\end{align*}
\end{remark}

We next turn to the functional setting and the notion of weak solution to \eqref{pa1}. In light of the expected mass conservation \eqref{zp1}, natural function spaces are the weighted $L^1$-space $L^1((0,\infty),xdx)$ or $L^1((0,\infty),(1+x)dx)$. In the latter, besides the finiteness of the total mass, we require the total number of particles in the system, which corresponds to the $L^1$-norm, to be finite. In fact, as for the classical coagulation-fragmentation equation, we shall use a scale of weighted $L^1$-spaces which we introduce now.

\bigskip

\noindent\textbf{Notation.} Given a non-negative measurable function $V$ on $(0,\infty)$, we set $X_V := L^1((0,\infty),V(x) dx)$ and 
\begin{equation*}
\|h\|_{X_V} := \int_0^\infty |h(x)| V(x)\ dx\,, \quad M_V(h) := \int_0^\infty h(x) V(x)\ dx\,, \qquad h\in X_V\,.
\end{equation*}
We also denote the positive cone of $X_V$ by $X_V^+$, while $X_{V,w}$ stands for the space $X_V$ endowed with its weak topology. When $V(x)=V_m(x) := x^m$, $x\in (0,\infty)$, for some $m\in\mathbb{R}$, we set $X_m := X_{V_m}$ and
\begin{equation*}
M_m(h) := M_{V_m}(h) = \int_0^\infty x^m h(x)\ dx\,, \qquad h\in X_m\,.
\end{equation*}
Note that $X_0 = L^1(0,\infty)$.

We may now state the definition of weak solution to \eqref{pa1} to be considered throughout the paper.

\begin{definition}\label{defa1}
Let $T \in (0,\infty]$ and consider a daughter distribution $b$ satisfying \eqref{pa2} and \eqref{pa4}. Given $f^{in}\in X_0 \cap X_1^+$, a weak solution to \eqref{pa1} on $[0,T)$ is a non-negative function
\begin{subequations}\label{pa5}
\begin{equation}
f\in C([0,T),X_{0,w}) \cap L^\infty((0,T),X_1^+) \label{pa5a}
\end{equation} 
such that
\begin{equation}
(s,x,y)\longmapsto K(x,y) f(s,x) f(s,y) \in L^1((0,t)\times (0,\infty)^2) \label{pa5b}
\end{equation}
and
\begin{equation}
\int_0^\infty \phi(x) (f(t,x)-f^{in}(x))\ dx = \frac{1}{2} \int_0^t \int_0^\infty \int_0^\infty \zeta_\phi(x,y) K(x,y) f(s,x) f(s,y)\ dydxds \label{pa5c}
\end{equation}
for all $t\in (0,T)$ and $\phi\in L^{\infty}(0,\infty)$, where
\begin{equation*}
\zeta_\phi(x,y) := \int_0^{x+y} \phi(z) b(z,x,y)\ dz - \phi(x) - \phi(y)\,, \qquad (x,y)\in (0,\infty)^2\,.
\end{equation*}
\end{subequations}

Moreover, a weak solution $f$ to \eqref{pa1} on $[0,T)$ is mass-conserving on $[0,T)$ if 
\begin{equation}
M_1(f(t)) = M_1(f^{in})\,, \qquad t\in [0,T)\,. \label{pa6}
\end{equation}
\end{definition}

Observe that \eqref{pa4} and the boundedness of $\phi$ ensure that $\zeta_\phi\in L^\infty((0,\infty)^2)$, so that the integral on the right hand side of \eqref{pa5c} is finite due to \eqref{pa5b}. We shall derive additional properties of weak solutions in the sense of Definition~\ref{defa1} in Section~\ref{sec2}.
	
\begin{remark}\label{rema1.5}
Let $f$ be a weak solution to \eqref{pa1} on $[0,T)$ for some $T\in (0,\infty]$ and consider $\phi\in L^\infty(0, \infty)$. Since $\zeta_\phi \in L^\infty((0,\infty)^2)$ by \eqref{pa4}, it readily follows from \eqref{pa5b} and \eqref{pa5c} that
\begin{equation*}
t\mapsto \int_0^\infty \phi(x) f(t,x)\ dx \in W_{loc}^{1,1}(0,T)
\end{equation*}
with
\begin{equation*}
\frac{d}{dt} \int_0^\infty \phi(x) f(t,x)\ dx = \frac{1}{2} \int_0^\infty \int_0^\infty \zeta_\phi(x,y) K(x,y) f(t,x) f(t,y)\ dydx \;\;\text{ for a.e. }\;\; t\in (0,T)\,.
\end{equation*}
We shall mainly use the above alternative formulation of \eqref{pa5c} in the sequel.
\end{remark}

We are left with specifying the class of collision kernels to be dealt with herein. We focus our attention on the following two-parameters family of explicit collision kernels $K$: There are 
\begin{subequations}\label{pa7}
\begin{equation}
\alpha \le \beta \le 1 \label{pa7b}
\end{equation}
such that
\begin{equation}
K(x,y) := x^\alpha y^\beta + x^\beta y^\alpha\,, \qquad (x,y)\in (0,\infty)^2\,. \label{pa7a}
\end{equation}
\end{subequations}
It is however likely that the analysis performed below equally applies to collision kernels being bounded from above and/or from below by multiples of the kernel defined by \eqref{pa7}. As already noticed in the literature \cite{ChRe90, ErPa07, KoKa00}, the case $\alpha=\beta$ is peculiar, as equation~\eqref{pa1} can be transformed to a linear fragmentation equation with a different scale. However, such a simplifying transformation does not seem to be available in general, and we thus use a different appraoch to study \eqref{pa1}.

We begin with the existence of mass-conserving weak solutions and prove the following result, which matches the outcome of \cite{ErPa07} for the particular case $\alpha=\beta$.

\begin{theorem}[Existence]\label{thma2}
Assume that $K$ is given by \eqref{pa7} and that $b$ satisfies \eqref{pa2}, \eqref{pa3}, and \eqref{pa4}, as well as
\begin{equation}
\int_0^x \bar{b}(z,x,y)^p\ dz \le \frac{B_p}{2} x^{1-p}\,, \qquad (x,y)\in (0,\infty)^2\,, \label{pa8} 
\end{equation}
for some $p\in (1,2)$ and $B_p>0$.

Let $f^{in}\in X_0 \cap X_1^+$ be an initial condition with $\varrho := M_1(f^{in})>0$ and set $\lambda := \alpha+\beta$. 
\begin{itemize}
	\item [(a)] If $\lambda\in [1,2]$, then there is at least one mass-conserving weak solution $f$ to \eqref{pa1} on $[0,\infty)$. 
	\item [(b)] If $\lambda\in [0,1)$ and $\alpha\ge 0$, then there is at least one mass-conserving weak solution $f$ to \eqref{pa1} on $[0,T_0(f^{in}))$, where
	\begin{equation}
	T_0(f^{in}) := \frac{M_0(f^{in})^{\lambda-1}}{(1-\lambda)(\beta_0-2) \varrho^\lambda} \in (0,\infty)\,. \label{pa9}
	\end{equation}
	\item [(c)] Furthermore, if $f^{in}\in X_m$ for some $m>1$, then the solution constructed above satisfies $f\in L^\infty((0,t),X_m)$ with $M_m(f(t)) \le M_m(f^{in})$ for all $t>0$ in case~(a) and all $t\in (0,T_0(f^{in}))$ in case~(b).
	\end{itemize}
\end{theorem}

Theorem~\ref{thma2} is restricted to daughter distribution functions $b$ satisfying \eqref{pa3} besides \eqref{pa2}, so that mass transfer during collisions is excluded. The general case seems to be more involved and will be investigated separately. In particular, the time monotonicity of superlinear moments stated in Theorem~\ref{thma2} is no longer valid when mass transfer is allowed. The proof of Theorem~\ref{thma2} relies on the $L^1$-weak compactness approach introduced in the pioneering work \cite{Stew89} dealing with the existence of weak solutions to the coagulation-fragmentation equation. The specific form of the collision-induced breakage equation \eqref{pa1} requires however to start from a different approximation, while special attention has to be  paid to the small size behaviour of the approximating sequence. We next show that Theorem~\ref{thma2}~(b) cannot be improved, in the sense that mass-conserving solutions cannot be extended to all times in that case.
	
\begin{theorem}[Finite time existence] \label{thma3}
Assume that $K$ is given by \eqref{pa7} with $\lambda=\alpha+\beta\in [0,1)$ and $\alpha\ge 0$. Assume also  that $b$ satisfies \eqref{pa2}, \eqref{pa3}, and \eqref{pa4} and that there is $\gamma_\lambda>1$ such that
\begin{equation}
\int_0^{x} z^\lambda \bar{b}(z,x,y)\ dz \ge \gamma_\lambda x^\lambda\,, \qquad (x,y)\in (0,\infty)^2\,. \label{pa10}
\end{equation} 
Let $f^{in}\in X_0 \cap X_1^+$ be an initial condition with $\varrho := M_1(f^{in})>0$ and consider $T>0$ such that there is a weak solution $f$ to \eqref{pa1} on $[0,T)$. Then 
\begin{equation*}
T \le \frac{1}{4(\gamma_\lambda-1) M_\lambda(f^{in})}\,.
\end{equation*}
\end{theorem}

The next step towards the well-posedness of \eqref{pa1} in a suitable framework is the uniqueness issue, for which we establish the following result. Not surprisingly, its proof follows the lines of the uniqueness results obtained for the coagulation-fragmentation equation, see \cite{BLL19, EMRR05, Giri13, Stew90b}, and the references therein.

\begin{proposition}[Uniqueness]\label{propa4}
Assume that $K$ is given by \eqref{pa7} with $\alpha\ge 0$ and that $b$ satisfies \eqref{pa2} and \eqref{pa4}. Consider an initial condition $f^{in}\in X_0 \cap X_1^+$ and $T>0$. There is at most one weak solution $f$ to \eqref{pa1} on $[0,T)$ such that 
\begin{equation}
M_{1+\beta}(f)\in L^1(0,t) \;\text{ for each }\; t\in (0,T)\,. \label{pa11}
\end{equation}
\end{proposition}

Gathering the outcome of Theorem~\ref{thma2} and Proposition~\ref{propa4} provides the global well-posedness of \eqref{pa1} when $0\le \alpha\le \beta \le 1$ and $\lambda\ge 1$ and its local well-posedness when $0\le \alpha\le \beta \le 1$ and $\lambda\in [0,1)$ in a suitable functional setting, as reported below.

\begin{corollary}[Well-posedness]\label{cora5}
Assume that $K$ is given by \eqref{pa7} with $\alpha\ge 0$ and that $b$ satisfies \eqref{pa2}, \eqref{pa3}, \eqref{pa4}, and \eqref{pa8}. Given $f^{in}\in X_0\cap X_{1+\beta}^+$, there is a unique mass-conserving weak solution $f$ to \eqref{pa1} on $[0,\infty)$ when $\lambda\ge 1$ and on $[0,T_0(f^{in}))$ when $\lambda\in [0,1)$ which is locally bounded in $X_{1+\beta}$. 
\end{corollary}

We finally show that, in contrast to Smoluchowski's coagulation equation, negative exponents in the collision kernel $K$ are not compatible with the existence of mass-conserving weak solutions to \eqref{pa1}, even locally in time. This striking feature is observed in \cite{ErPa07} in the particular case $\alpha=\beta<0$. We provide here a rigorous proof which borrows arguments from the study of the occurrence of instantaneous gelation in Smoluchowski's coagulation equation \cite{BLL19, CadC92, vanD87c}.

\begin{theorem}[Non-existence]\label{thma6}
Assume that $K$ is given by \eqref{pa7} and that there is $\nu\in (-1,0]$ such that \begin{equation}
b(z,x,y) = (\nu+2) z^\nu \mathbf{1}_{(0,x)}(z) x^{-\nu-1} + (\nu+2) z^\nu \mathbf{1}_{(0,y)}(z) y^{-\nu-1}\,, \qquad (z,x,y)\in (0,\infty)^3\,. \label{pa12}
\end{equation}
Consider an initial condition $f^{in}\in X_0 \cap X_1^+$  with $\varrho := M_1(f^{in})>0$ and assume further that $f^{in}\in X_{m_0}$ for some $m_0>1$. If $\alpha<0$, then there is no mass-conserving weak solution to \eqref{pa1}.
\end{theorem}

\begin{remark}\label{remz}
Observe that, when $\nu\in (-1,0]$, the daughter distribution $b$ given by \eqref{pa12} satisfies \eqref{pa2}, \eqref{pa3} (with $\bar{b}=\bar{b}_\nu$ defined in \eqref{zp4}), and \eqref{pa4} (with $\beta_0=2(\nu+2)/(\nu+1)$). Moreover, still for $\nu\in (-1,0]$, $\bar{b}=\bar{b}_\nu$ satisfies \eqref{pa8} with $p\in (1,2)\cap (1,1/|\nu|)$ and $B_p=2(\nu+2)^p/(1+\nu p)$, as well as \eqref{pa10} with $\gamma_\lambda = (\nu+2)/(\nu+\lambda+1)$ for $\lambda\in [0,1)$. Therefore, Theorems~\ref{thma2} and~\ref{thma3} apply to this particular choice of daughter distribution function. 
\end{remark}

\bigskip

 We now describe the content of this paper. The next section is devoted to the derivation of additional properties of weak solutions to \eqref{pa1} in the sense of Definition~\ref{defa1}, including strong continuity with respect to time (Proposition~\ref{propb1}), a criterion for mass conservation (Proposition~\ref{propb2}), and some tail control induced by the assumption \eqref{pa3} on $b$ (Lemma~\ref{lemb3}). In Section~\ref{sec3}, we establish the well-posedness of \eqref{pa1} under the sole assumptions \eqref{pa2} and \eqref{pa4} on $b$, provided the collision kernel is bounded from above by $k_0 xy/(x+y)$. This assumption implies that $\mathcal{B}$ is a Lipschitz continuous map on $X_1$ and on $X_0\cap X_1$ and the well-posedness of \eqref{pa1} in that case is proved with Banach's fixed point theorem. Section~\ref{sec4} is devoted to the proof of Theorem~\ref{thma2} which relies on a compactness approach. Estimates for large and small sizes are first derived to exclude escape of matter as $x\to\infty$ or $x\to 0$ and are supplemented with a uniform integrability estimate which prevents concentration at a finite size. Thanks to Dunford-Pettis' theorem, these estimates guarantee the weak compactness of the approximating sequence with respect to the size variable. The compactness with respect to time is next obtained as a consequence of a time equicontinuity estimate. We are then left with passing to the limit as the approximating parameter converges to zero to complete the proof. The uniqueness and finite time existence issues are next discussed in Sections~\ref{sec5} and~\ref{sec6}, respectively. We end up the paper with the proof of the non-existence result in Section~\ref{sec7}. As in \cite{BLL19, CadC92, vanD87c}, the main step is to show that, if $f$ is a mass-conserving weak solution to \eqref{pa1} on $[0,T)$ and $\alpha<0$, then $M_m(f(t))$ must be finite for all $t\in [0,T)$ and $m\in (-\infty,1)$. This implies in particular that, for such a solution to exist, the initial condition $f^{in}$ should necessarily belong to $X_m$ for any $m\in (-\infty,1)$. But additional information is actually provided by this step which leads to $T=0$, thereby excluding the existence of mass-conserving weak solutions, even for short times.

\section{Basic properties}\label{sec2}

Let us first recall that the regularity properties of weak solutions listed in Definition~\ref{defa1} guarantee that this definition is meaningful. Indeed, if $\phi\in L^{\infty}(0,\infty)$ and $b$ satisfies \eqref{pa2} and \eqref{pa4}, then

\begin{equation}
|\zeta_{\phi}(x,y)| \le \int_0^{x+y} |\phi(z)| b(z,x,y)\ dz + |\phi(x)| + |\phi(y)| \le (\beta_0+2) \|\phi\|_{L^\infty(0,\infty)} \label{zp2}
\end{equation}
for $(x,y)\in (0,\infty)^2$. This estimate, along with \eqref{pa5a} and \eqref{pa5b}, ensures that all the terms involved in \eqref{pa5c} are well-defined.

\subsection{Strong continuity}\label{sec2.1}

We first show that the time continuity of weak solutions to \eqref{pa1} is actually with respect to the norm-topology of $X_0$.

\begin{proposition}\label{propb1}
Let $T\in (0,\infty]$, $f^{in}\in X_0\cap X_1^+$, and consider a weak solution $f$ to \eqref{pa1} on $[0,T)$. Then $f\in C([0,T),X_0)$.
\end{proposition}

\begin{proof}
Let $\phi\in L^\infty(0,\infty)$. It readily follows from \eqref{pa5c} and \eqref{zp2} that, for $0\le t_1 < t_2 < T$, 
\begin{align*}
& \left| \int_0^\infty \phi(x) (f(t_2,x)-f(t_1,x))\ dx \right| \\
& \hspace{1cm} \le (\beta_0+2) \|\phi\|_{L^\infty(0,\infty)} \int_{t_1}^{t_2} \int_0^\infty \int_0^\infty K(x,y) f(s,x) f(s,y)\ dydxds\,.
\end{align*}
Hence, since $L^\infty(0,\infty)$ is the dual space of $L^1(0,\infty)$, 
\begin{equation*}
\int_0^\infty |f(t_2,x)-f(t_1,x)| \ dx \le (\beta_0+2) \int_{t_1}^{t_2} \int_0^\infty \int_0^\infty K(x,y) f(s,x) f(s,y)\ dydxds\,,
\end{equation*}
and the claimed continuity of $f$ readily follows from \eqref{pa5b}.
\end{proof}

\subsection{Mass conservation}\label{sec2.2}

We next provide a criterion for a weak solution to be mass-conserving.

\begin{proposition}\label{propb2}
Let $T\in (0,\infty]$, $f^{in}\in X_0\cap X_1^+$, and consider a weak solution $f$ to \eqref{pa1} on $[0,T)$ such that
\begin{equation}
s\longmapsto \int_0^\infty \int_0^\infty (x+y) K(x,y) f(s,x) f(s,y)\ dy dx \in L^1(0,t)
\label{pb1}
\end{equation}
for all $t\in (0,T)$. Then $f$ is mass-conserving on $[0,T)$.
\end{proposition}

\begin{proof}
For $A>0$ and $x\in (0,\infty)$, we set $\phi_A(x) := x \mathbf{1}_{(0,A)}(x)$, and observe that, by \eqref{pa2}:
\begin{itemize}
	\item if $(x,y)\in (0,A)^2$ with $x+y\in (0,A)$, then $\zeta_{\phi_A}(x,y)=0$;
	\item if $(x,y)\in (0,A)^2$ with $x+y >A$, then 
	\begin{align*}
	\zeta_{\phi_A}(x,y) & = \int_0^A z b(z,x,y)\ dz - x - y = - \int_A^{x+y} z b(z,x,y)\ dz \in (-(x+y),0)\,;
	\end{align*}
	\item if $(x,y)\in (0,A)\times (A,\infty)$, then
	\begin{equation*}
	\zeta_{\phi_A}(x,y) = y- \int_A^{x+y} z b(z,x,y)\ dz \in (-x,y)\,;
	\end{equation*}
	\item if $(x,y)\in (A,\infty)\times (0,A)$, then
	\begin{equation*}
	\zeta_{\phi_A}(x,y) = x - \int_A^{x+y} z b(z,x,y)\ dz \in (-y,x)\,;
	\end{equation*}
	\item if $(x,y)\in (A,\infty)^2$, then
	\begin{equation*}
	\zeta_{\phi_A}(x,y) = x+y - \int_A^{x+y} z b(z,x,y)\ dz \in (0,x+y)\,.
	\end{equation*}
\end{itemize}
Overall, we have shown that 
\begin{equation*}
-(x+y) \le \zeta_{\phi_A}(x,y) \le (x+y)\,, \qquad (x,y)\in (0,\infty)^2\,,
\end{equation*}
which gives, together with \eqref{pa5c},
\begin{equation*}
\int_0^\infty \phi_A(x) f(t,x)\ dx \le \int_0^\infty \phi_A(x) f^{in}(x)\ dx + I(t) \le M_1(f^{in}) + I(t)\,, \quad t\in [0,T)\,,
\end{equation*}
with
\begin{equation*}
I(t) := \frac{1}{2} \int_0^t \int_0^\infty \int_0^\infty (x+y) K(x,y) f(s,x) f(s,y)\ dydxds\,.
\end{equation*}
Since the right hand side of the above inequality is finite by \eqref{pb1} and does not depend on $A$, we may let $A\to\infty$ and deduce from Fatou's lemma that $M_1(f(t))<\infty$ with $M_1(f(t)) \le M_1(f^{in}) + I(t)$ for all $t\in [0,T)$. Next, since
\begin{equation*}
\lim_{A\to \infty} \zeta_{\phi_A}(x,y) = 0\,, \qquad (x,y)\in (0,\infty)^2\,,
\end{equation*}
and
 \begin{equation*}
|\zeta_{\phi_A}(x,y)| K(x,y) f(s,x) f(s,y)\le (x+y) K(x,y) f(s,x) f(s,y)
\end{equation*}
for $(s,x,y)\in (0,t)\times (0,\infty)^2$, we infer from \eqref{pb1} and Lebesgue's dominated convergence theorem that
\begin{equation*}
\lim_{A\to\infty} \int_0^t \int_0^\infty \int_0^\infty \zeta_{\phi_A}(x,y) K(x,y) f(s,x) f(s,y)\ dydxds = 0\,.
\end{equation*}
Consequently, recalling \eqref{pa5c} and the property $f(t)\in X_1$ for $t\in [0,T)$, we conclude that
\begin{equation*}
M_1(f(t)) = \lim_{A\to\infty} \int_0^\infty \phi_A(x) f(t,x)\ dx = \lim_{A\to\infty} \int_0^\infty \phi_A(x) f^{in}(x)\ dx = M_1(f^{in})\,,
\end{equation*}
and the proof is complete.
\end{proof}

\subsection{Tail control}\label{sec2.3}

 We next report the time monotonicity of the cumulative distribution function when $b$ satisfies \eqref{pa3}, already observed in \cite{ErPa07}. Such a property obviously plays an important role in the control of the behaviour for large sizes and the proof of mass conservation.
	
\begin{lemma}[\cite{ErPa07}]\label{lemb3} 
Let $T\in (0,\infty]$, $f^{in}\in X_0\cap X_1^+$, and consider a mass-conserving weak solution $f$ to \eqref{pa1} on $[0,T)$. If $b$ is of the form \eqref{pa3}, then 
\begin{equation*}
\int_x^\infty y f(t,y)\ dy \le \int_x^\infty y f^{in}(y)\ dy\,, \qquad (t,x)\in [0,T)\times (0,\infty)\,.
\end{equation*}
\end{lemma}

\begin{proof}
As in the proof of Proposition~\ref{propb2}, we set $\phi_A(x) := x \mathbf{1}_{(0,A)}(x)$ for $A>0$ and $x\in (0,\infty)$. Owing to  \eqref{pa3}:
\begin{itemize}
	\item if $(x,y)\in (0,A)^2$, then 
	\begin{align*}
	\zeta_{\phi_A}(x,y) & = \int_0^x z \bar{b}(z,x,y)\ dz + \int_0^y z \bar{b}(z,y,x)\ dz - x - y = 0\,;
	\end{align*}
	\item if $(x,y)\in (0,A)\times (A,\infty)$, then
	\begin{equation*}
	\zeta_{\phi_A}(x,y) = \int_0^x z \bar{b}(z,x,y)\ dz + \int_0^A z \bar{b}(z,y,x)\ dz - x = \int_0^A z \bar{b}(z,y,x)\ dz \ge 0\,;
	\end{equation*}
	\item if $(x,y)\in (A,\infty)\times (0,A)$, then
	\begin{equation*}
	\zeta_{\phi_A}(x,y) = \int_0^A z \bar{b}(z,x,y)\ dz + \int_0^y z \bar{b}(z,y,x)\ dz - y = \int_0^A z \bar{b}(z,x,y)\ dz \ge 0\,;	\end{equation*}
	\item if $(x,y)\in (A,\infty)^2$, then
	\begin{equation*}
	\zeta_{\phi_A}(x,y) = \int_0^A z \bar{b}(z,x,y)\ dz + \int_0^A z \bar{b}(z,y,x)\ dz \ge 0\,.
	\end{equation*}
\end{itemize}
Consequently, $\zeta_{\phi_A}(x,y) \ge 0$ for all $(x,y)\in (0,\infty)^2$ and, since both $K$ and $f$ are non-negative by \eqref{pa7} and Definition~\ref{defa1}, we infer from \eqref{pa5c} that
\begin{equation*}
\int_0^A x f(t,x)\ dx \ge \int_0^A x f^{in}(x)\ dx\,, \qquad t\in [0,T)\,.
\end{equation*}
Combining the above inequality with the conservation of mass completes the proof. 
\end{proof}

A similar monotonicity with respect to time actually holds true for superlinear moments.

\begin{lemma}\label{lemb4} 	
Let $T\in (0,\infty]$, $f^{in}\in X_0\cap X_1^+$, and consider a mass-conserving weak solution $f$ to \eqref{pa1} on $[0,T)$. If $b$ is of the form \eqref{pa3} and $f^{in}\in X_m$ for some $m>1$, then $f(t)\in X_m$ and $M_m(f(t))\le M_m(f^{in})$ for all $t\in (0,T)$.
\end{lemma}

\begin{proof}
Let $A>0$ and set $\phi(x) = \left( x^m - A^{m-1} x \right) \mathbf{1}_{(0,A)}(x)$ for $x\in (0,\infty)$. By \eqref{pa3},
\begin{itemize}
	\item if $(x,y)\in (0,A)^2$, then
	\begin{align*}
	\zeta_{\phi}(x,y) & = \int_0^x \left( z^m - A^{m-1} z \right) \bar{b}(z,x,y)\ dz + \int_0^y \left( z^m - A^{m-1} z \right) \bar{b}(z,y,x)\ dz \\
	& \qquad - x^m + A^{m-1} x - y^m + A^{m-1} y \\
	& \le \left( x^{m-1} - A^{m-1} \right) \int_0^x z \bar{b}(z,x,y)\ dz + \left( y^{m-1} - A^{m-1} \right) \int_0^y z \bar{b}(z,y,x)\ dz \\
	& \qquad - x^m + A^{m-1} x - y^m + A^{m-1} y \\
	& = 0\,;
	\end{align*}
	\item if $(x,y)\in (0,A)\times (A,\infty)$, then
	\begin{align*}
	\zeta_{\phi}(x,y) & = \int_0^x \left( z^m - A^{m-1} z \right) \bar{b}(z,x,y)\ dz + \int_0^A \left( z^m - A^{m-1} z \right) \bar{b}(z,y,x)\ dz \\
	& \qquad - x^m + A^{m-1} x \\
	& \le \left( x^{m-1} - A^{m-1} \right) \int_0^x z \bar{b}(z,x,y)\ dz - x^m + A^{m-1} x \\
	& = 0\,;
	\end{align*}
	\item if $(x,y)\in (A,\infty)\times (0,A)$, then
	\begin{align*}
	\zeta_{\phi}(x,y) & = \int_0^A \left( z^m - A^{m-1} z \right) \bar{b}(z,x,y)\ dz + \int_0^y \left( z^m - A^{m-1} z \right) \bar{b}(z,y,x)\ dz \\
	& \qquad - y^m + A^{m-1} y \\
	& \le \left( y^{m-1} - A^{m-1} \right) \int_0^y z \bar{b}(z,y,x)\ dz - y^m + A^{m-1} y \\
	& = 0\,;
	\end{align*}
	\item if $(x,y)\in (A,\infty)^2$, then
	\begin{align*}
	\zeta_{\phi}(x,y) & = \int_0^A \left( z^m - A^{m-1} z \right) \bar{b}(z,x,y)\ dz + \int_0^A \left( z^m - A^{m-1} z \right) \bar{b}(z,y,x)\ dz \\
	& \le 0\,.
	\end{align*}
\end{itemize}
Consequently, taking into account that $\phi$ belongs to $L^\infty(0,\infty)$, it follows from \eqref{pa5c} that
\begin{equation*}
\int_0^A \left( x^m - A^{m-1} x\right) [f(t,x)-f^{in}(x)]\ dx \le 0\,, \qquad t\in (0,T) \,.
\end{equation*}
Combining the above inequality with the mass-conserving property of $f$ leads us to
\begin{equation*}
\int_0^\infty \min\{x^m, A^{m-1}x\} f(t,x)\ dx \le \int_0^\infty \min\{x^m, A^{m-1}x\} f^{in}(x)\ dx \le M_m(f^{in})\,, \qquad t\in (0,T)\,.
\end{equation*}
In particular, for all $A>0$,
\begin{equation*}
	\int_0^A x^m f(t,x)\ dx \le  M_m(f^{in})\,, \qquad t\in (0,T)\,.
\end{equation*}
Letting $A\to\infty$ in the previous inequality readily entails the claimed result.
\end{proof}

\section{Well-posedness by a fixed-point approach}\label{sec3}

The starting point of the compactness method to be designed in Section~\ref{sec4} is the well-posedness of \eqref{pa1} in $X_0\cap X_1^+$ for a suitable class of collision kernels which we describe now.

\begin{proposition}\label{propc1}
Consider $f^{in}\in X_0\cap X_1^+$ and assume that $b$ satisfies \eqref{pa2} and \eqref{pa4} and that there is $k_0>0$ such that
\begin{equation}
K(x,y) \le k_0 \frac{xy}{x+y}\,, \qquad (x,y)\in (0,\infty)^2\,. \label{pb2}
\end{equation}
Then \eqref{pa1} has a unique strong solution $f\in C^1([0,\infty), X_0 \cap X_1^+)$ which satisfies
\begin{subequations}\label{pb3}
\begin{equation}
\frac{d}{dt} \int_0^\infty \phi(x) f(t,x)\ dx = \frac{1}{2} \int_0^\infty \int_0^\infty \zeta_\phi(x,y) K(x,y) f(t,x) f(t,y)\ dydx \label{pb3a}
\end{equation}
for all $\phi\in L^\infty(0,\infty)$ and $t\ge 0$ and
\begin{equation}
M_1(f(t)) = M_1(f^{in})\,, \qquad t\ge 0\,. \label{pb3b}
\end{equation}
In particular, it is also a mass-conserving weak solution to \eqref{pa1} on $[0,\infty)$ and, for $t\ge 0$, 
\begin{equation}
\frac{d}{dt} M_0(f(t)) = \frac{1}{2} \int_0^\infty \int_0^\infty \left[ \int_0^{x+y} b(z,x,y)\ dz - 2 \right] K(x,y) f(t,x) f(t,y)\ dydx\,. \label{pb3c}
\end{equation}
\end{subequations}
\end{proposition}

The proof of Proposition~\ref{propc1} relies on Banach's fixed-point theorem and requires suitable Lipschitz properties of the nonlinear fragmentation operator $\mathcal{B}$ defined in \eqref{pa1c}. To proceed further, we split $\mathcal{B}=\mathcal{G}-\mathcal{D}$, with gain term
\begin{equation*}
\mathcal{G}f(x) := \frac{1}{2} \int_x^\infty \int_0^y b(x,y-z,z) K(y-z,z) f(y-z) f(z)\ dzdy \,, \qquad x\in (0,\infty)\,,
\end{equation*}
and loss term 
\begin{equation*}
\mathcal{D}f(x) := \int_0^\infty K(x,y) f(x) f(y)\ dy\,, \qquad x\in (0,\infty)\,.
\end{equation*}

We begin with a useful identity involving $\mathcal{G}$.

\begin{lemma}\label{lemc2}
Assume that $b$ and $K$ satisfy \eqref{pa2}, \eqref{pa4}, and \eqref{pb2}, respectively. If $\phi\in L^\infty(0,\infty)$ and $f\in X_0\cap X_1^+$, then $\phi \mathcal{G}(f)$ belongs to $X_0$ and
\begin{equation}
\int_0^\infty \phi(x) \mathcal{G}f(x)\ dx = \frac{1}{2} \int_0^\infty \int_0^\infty \int_0^{x+y} \phi(z) b(z,x,y) K(x,y) f(x) f(y)\ dzdydx\,. \label{pb4}
\end{equation}
\end{lemma}

\begin{proof}
We first infer from \eqref{pa4} and \eqref{pb2} that
\begin{align*}
& \int_0^\infty \int_0^\infty \int_0^{x+y} |\phi(z)| b(z,x,y) K(x,y) f(x) f(y)\ dzdydx \\
& \hspace{1cm} \le \|\phi\|_{L^\infty(0,\infty)} \int_0^\infty \int_0^\infty \int_0^{x+y} b(z,x,y) K(x,y) f(x) f(y)\ dzdydx \\
& \hspace{1cm} \le \beta_0 k_0 \|\phi\|_{L^\infty(0,\infty)} \int_0^\infty \int_0^\infty \frac{xy}{x+y} f(x) f(y)\ dydx \\
& \hspace{1cm} \le \beta_0 k_0 \|\phi\|_{L^\infty(0,\infty)} \int_0^\infty \int_0^\infty x f(x) f(y)\ dydx \\
& \hspace{1cm} = \beta_0 k_0 \|\phi\|_{L^\infty(0,\infty)} \|f\|_{X_0} \|f\|_{X_1}\,,
\end{align*}
so that
\begin{equation*}
(x,y,z) \longmapsto \phi(z) b(z,x,y) \mathbf{1}_{(0,x+y)}(z) K(x,y) f(x) f(y) \in L^1((0,\infty)^3)\,.
\end{equation*}
It then follows from Fubini's theorem that
\begin{align*}
& \frac{1}{2} \int_0^\infty \int_0^\infty \int_0^{y+z} \phi(x) b(x,y,z) K(y,z) f(y) f(z)\ dxdydz \\
& \hspace{1cm} = \frac{1}{2} \int_0^\infty \int_z^\infty \int_0^{y_*} \phi(x) b(x,y_*-z,z) K(y_*-z,z) f(y_*-z) f(z)\ dxdy_*dz \\
& \hspace{1cm} = \frac{1}{2} \int_0^\infty \int_0^y \int_0^y \phi(x) b(x,y-z,z) K(y-z,z) f(y-z) f(z)\ dxdzdy \\
& \hspace{1cm} = \frac{1}{2} \int_0^\infty \int_x^\infty \int_0^y \phi(x) b(x,y-z,z) K(y-z,z) f(y-z) f(z)\ dzdydx \\
& \hspace{1cm} = \int_0^\infty \phi(x) \mathcal{G}f(x)\ dx\,,
\end{align*}
which proves \eqref{pb4}. 
\end{proof}

We next report Lipschitz properties of $\mathcal{G}$ and $\mathcal{D}$.

\begin{lemma}\label{lemc3}
Assume that $b$ and $K$ satisfy \eqref{pa2}, \eqref{pa4}, and \eqref{pb2}, respectively. Then $\mathcal{G}$ and $\mathcal{D}$ are locally Lipschitz continuous on $X_0\cap X_1$.
\end{lemma}

\begin{proof}
Let $(f,g)\in (X_0\cap X_1)^2$. As in the proof of Lemma~\ref{lemc2}, it follows from \eqref{pa2}, \eqref{pb2}, and Fubini-Tonelli's theorem that
\begin{align*}
& \|\mathcal{G}f - \mathcal{G}g\|_{X_1} \\
& \qquad  \le \frac{1}{2} \int_0^\infty \int_x^\infty \int_0^y x b(x,y-z,z) K(y-z,z) |f(y-z) f(z) - g(y-z) g(z)|\ dzdydx \\
& \qquad =  \frac{1}{2} \int_0^\infty \int_0^\infty \int_0^{y+z} x b(x,y,z) K(y,z) |f(y) f(z) - g(y) g(z)|\ dxdzdy \\
& \qquad \le  \frac{1}{2} \int_0^\infty \int_0^\infty (y+z) K(y,z) \left(|f(y)| |(f-g)(z)| + |g(z)| |(f-g)(y)| \right)\ dzdy \\
& \qquad \le \frac{k_0}{2} \int_0^\infty \int_0^\infty yz \left(|f(y)| |(f-g)(z)| + |g(z)| |(f-g)(y)| \right)\ dzdy \\
& \qquad = \frac{k_0}{2} \left( \|f\|_{X_1} + \|g\|_{X_1} \right) \|f-g\|_{X_1}\,.
\end{align*}
Similarly, by \eqref{pb2} and the symmetry of $K$,
\begin{align*}
\|\mathcal{D}f - \mathcal{D}g\|_{X_1} & \le \int_0^\infty \int_0^\infty y K(y,z) |f(y) f(z) - g(y) g(z)|\ dzdy \\
& = \frac{1}{2} \int_0^\infty \int_0^\infty (y+z) K(y,z) |f(y) f(z) - g(y) g(z)|\ dzdy \\
& \le \frac{k_0}{2} \int_0^\infty \int_0^\infty yz |f(y) f(z) - g(y) g(z)|\ dzdy \\
& \le \frac{k_0}{2} \left( \|f\|_{X_1} + \|g\|_{X_1} \right) \|f-g\|_{X_1}\,.
\end{align*}
Consequently, $\mathcal{G}$ and $\mathcal{D}$ are locally Lipschitz continuous on $X_1$.

We next deduce from \eqref{pa4}, \eqref{pb2}, and Fubini-Tonelli's theorem that
\begin{align*}
& \|\mathcal{G}f - \mathcal{G}g\|_{X_0} \\
& \qquad \le  \frac{1}{2} \int_0^\infty \int_0^\infty \int_0^{y+z} b(x,y,z) K(y,z) |f(y) f(z) - g(y) g(z)|\ dxdzdy \\
& \qquad \le  \frac{\beta_0}{2} \int_0^\infty \int_0^\infty K(y,z) \left(|f(y)| |(f-g)(z)| + |g(z)| |(f-g)(y)| \right)\ dzdy \\
& \qquad \le \frac{k_0 \beta_0}{2} \int_0^\infty \int_0^\infty \left( y |f(y)| |(f-g)(z)| + z |g(z)| |(f-g)(y)| \right)\ dzdy \\
& \qquad = \frac{k_0 \beta_0}{2} \left( \|f\|_{X_1} + \|g\|_{X_1} \right) \|f-g\|_{X_0}\,.
\end{align*}
Similarly, by \eqref{pb2},
\begin{align*}
\|\mathcal{D}f - \mathcal{D}g\|_{X_0} & \le \int_0^\infty \int_0^\infty K(y,z) \left( |f(y)| |f(z) - g(z)| + |g(z)| |f(y) - g(y)| \right)\ dzdy \\
& \le k_0 \int_0^\infty \int_0^\infty \left( y |f(y)| |f(z) - g(z)| + z |g(z)| |f(y) - g(y)| \right)\ dzdy \\
& = k_0 \left( \|f\|_{X_1} + \|g\|_{X_1} \right) \|f-g\|_{X_0}\,.
\end{align*}
We have thus proved that $\mathcal{G}$ and $\mathcal{D}$ are locally Lipschitz continuous on $X_0$.
\end{proof}

We are now in a position to prove Proposition~\ref{propc1}.

\begin{proof}[Proof of Proposition~\ref{propc1}]
Let $f^{in}\in X_0\cap X_1^+$ and consider the initial value problem 
\begin{subequations}\label{pb5}
\begin{align}
\partial_t f & = \left[ \mathcal{G}f \right]_+ - \mathcal{D}f\,, \qquad t>0\,, \label{pb5a} \\
f(0) & = f^{in}\,, \label{pb5b}
\end{align}
\end{subequations}
where $\left[ \cdot\right]_+$ denotes the positive part. Since both $\mathcal{G}$ and $\mathcal{D}$ are locally Lipschitz continuous on $X_0\cap X_1$ by Lemma~\ref{lemc3} and taking the positive part is a contraction on $\mathbb{R}$, we infer from the Picard-Lindel\"of theorem, see, e.g., \cite[Theorems~7.6 \&~9.2]{Aman90}, that there is a unique solution $f\in C^1([0,T_{max}),X_0\cap X_1)$ to \eqref{pb5} defined on a maximal existence interval $[0,T_{max})$ with $T_{max}\in (0,\infty]$. Moreover, since $\mathcal{G}$ and $\mathcal{D}$ map bounded sets of $X_0\cap X_1$ into bounded sets of $X_0\cap X_1$, the following alternative holds true: either $T_{max}=\infty$, or
\begin{equation}
T_{max} < \infty \;\;\text{ and }\;\; \lim_{t\to T_{max}} \left( \|f(t)\|_{X_0} + \|f(t)\|_{X_1} \right)= \infty\,. \label{pb6}
\end{equation} 
see, e.g., \cite[Remark~7.10~(b)]{Aman90}. 

It first follows from \eqref{pb5a} that, for $t\in (0,T_{max})$,
\begin{align*}
\partial_t [-f]_+(t,x)  & = - \mathrm{sign}_+ (-f(t,x)) \left( \left[ \mathcal{G}f(t,x) \right]_+ - \mathcal{D}f(t,x) \right) \\
& \le - \left( \int_0^\infty K(x,y) f(t,y)\ dy \right) [-f]_+(t,x)\,.
\end{align*}
Therefore, thanks to \eqref{pb2}, 
\begin{align*}
\frac{d}{dt} \|[-f]_+(t)\|_{X_1} & \le \int_0^\infty \int_0^\infty x K(x,y) |f(t,y)| [-f]_+(t,x) \ dydx \\
& \le k_0\int_0^\infty \int_0^\infty \frac{x}{x+y} xy |f(t,y)| [-f]_+(t,x) \ dydx \\
& \le k_0 \|f(t)\|_{X_1} \|[-f]_+(t) \|_{X_1}\,.
\end{align*}
Hence, 
\begin{equation*}
\|[-f]_+(t)\|_{X_1} \le \|[-f^{in}]_+\|_{X_1} \exp\left\{ k_0 \int_0^t \|f(s)\|_{X_1}\ ds \right\} = 0\,, \qquad t\in [0,T_{max})\,;
\end{equation*}
that is, $f(t)\in X_1^+$ for all $t\in [0,T_{max})$. This non-negativity property entails that $\left[ \mathcal{G}f\right]_+ = \mathcal{G}f$ and it follows from \eqref{pb5} that $f\in C^1([0,T_{max}),X_0\cap X_1)$ solves
\begin{subequations}\label{pb7}
\begin{align}
\partial_t f & = \mathcal{G}f - \mathcal{D}f = \mathcal{B}f\,, \qquad t>0\,, \label{pb7a} \\
f(0) & = f^{in}\,. \label{pb7b}
\end{align}
\end{subequations}
Since 
\begin{align*}
\int_0^\infty \int_0^\infty K(x,y) f(s,x) f(s,y)\ dydx & \le k_0 \|f(s)\|_{X_0} \|f(s)\|_{X_1} \\ 
& \le k_0 \sup_{s\in [0,t]}\{ \|f(s)\|_{X_0} \|f(s)\|_{X_1}\} 
\end{align*}
for $s\in (0,t)$ and $t\in (0,T_{max})$, it readily follows from \eqref{pb4} and \eqref{pb7} that $f$ satisfies \eqref{pb3a} and is a weak solution to \eqref{pa1} on $[0,T_{max})$. In addition, we infer from \eqref{pb2} that, for $t\in (0,T_{max})$ and $s\in (0,t)$, 
\begin{equation*}
\int_0^\infty \int_0^\infty (x+y) K(x,y) f(s,x) f(s,y)\ dydx \le k_0 \|f(s)\|_{X_1}^2 \le k_0 \sup_{s\in [0,t]}\{ \|f(s)\|_{X_1}^2\} \,, 
\end{equation*}
so that, according to Proposition~\ref{propb2}, $f$ is a mass-conserving weak solution to \eqref{pa1} on $[0,T_{max})$. Therefore, 
\begin{equation}
\|f(t)\|_{X_1} = M_1(f(t)) = M_1(f^{in}) = \|f^{in}\|_{X_1}\,, \qquad t\in [0,T_{max})\,. \label{pb8}
\end{equation}
Also, for $t\in [0,T_{max})$, we deduce from \eqref{pb3a} with $\phi\equiv 1$ that
\begin{equation*}
\frac{d}{dt} M_0(f(t)) = \frac{1}{2} \int_0^\infty \int_0^\infty \left[ \int_0^{x+y} b(z,x,y)\ dz - 2 \right] K(x,y) f(t,x) f(t,y)\ dydx\,.
\end{equation*}
Owing to \eqref{pa4}, \eqref{pb2}, and \eqref{pb8}, we further obtain
\begin{align*}
\frac{d}{dt} M_0(f(t)) & \le \frac{\beta_0}{2} \int_0^\infty \int_0^\infty K(x,y) f(t,x) f(t,y)\ dydx \\
& \le \frac{k_0 \beta_0}{2} \int_0^\infty \int_0^\infty \frac{xy}{x+y} f(t,x) f(t,y)\ dydx \\
& \le \frac{k_0 \beta_0}{2} M_1(f(t)) M_0(f(t)) = \frac{k_0 \beta_0}{2} M_1(f^{in}) M_0(f(t))\,.
\end{align*}
Consequently,
\begin{equation*}
\|f(t)\|_{X_0} = M_0(f(t)) \le M_0(f^{in}) e^{k_0\beta_0 t/2}\,, \qquad t\in [0,T_{max})\,,
\end{equation*}
which, together with \eqref{pb8}, rules out \eqref{pb6} and implies $T_{max}=\infty$. 
\end{proof}

\section{Existence by a compactness approach}\label{sec4}

In this section, we assume that the daughter distribution function $b$ satisfies \eqref{pa2}, \eqref{pa3}, \eqref{pa4}, and \eqref{pa8}, while $K$ is given by \eqref{pa7} with \begin{equation}
\alpha\ge 0\,. \label{pc1}
\end{equation}
Also, let $f^{in}\in X_0 \cap X_1^+$ with $\varrho:= M_1(f^{in})>0$. Since $f^{in}\in X_0$, a variant of the de la Vall\'ee Poussin theorem, see \cite[Theorem~7.1.6]{BLL19}, ensures that there is a function $\Phi\in C^1([0,\infty))$ endowed with the following properties: $\Phi$ is convex, $\Phi(0)=\Phi'(0)=0$, $\Phi'$ is concave and positive on $(0,\infty)$, 
\begin{equation}
\mathcal{I} := \int_0^\infty \Phi(f^{in}(x))\ dx < \infty\,, \label{pc2}
\end{equation}
\begin{equation}
\lim_{r\to\infty} \Phi'(r) = \lim_{r\to\infty} \frac{\Phi(r)}{r} = \infty\,, \label{pc3}
\end{equation}
and, for all $\mu\in (1,2]$,
\begin{equation}
\lim_{r\to\infty} \frac{\Phi'(r)}{r^{\mu-1}} = \lim_{r\to\infty} \frac{\Phi(r)}{r^\mu} = 0\,. \label{pc4}
\end{equation}
Another variant of the de la Vall\'ee-Poussin theorem, recalled in Lemma~\ref{leap1}, and the integrability of $f^{in}$ guarantee that there is a non-negative convex and non-increasing function $\Phi_0\in C^1((0,\infty))$ such that
\begin{equation}
M_{\Phi_0}(f^{in}) = \int_0^\infty \Phi_0(x) f^{in}(x)\ dx < \infty \label{pc4.1}
\end{equation}
and
\begin{equation}
\lim_{x\to 0} \Phi_0(x) = \infty\,, \qquad \lim_{x\to 0} x^{(p-1)/2p} \Phi_0(x) = 0\,, \qquad x \mapsto x^{(p-1)/2p} \Phi_0(x) \;\text{ is non-decreasing}\,, \label{pc4.2}
\end{equation}
the parameter $p\in (1,2)$ being defined in \eqref{pa8}.

We now introduce a suitably designed approximation of \eqref{pa1}, to which we may apply the analysis performed in the previous section. Specifically, let $n\ge 1$ be an integer and define
\begin{equation}
K_n(x,y) := K(x,y) \mathbf{1}_{(1/n,n)}(x) \mathbf{1}_{(1/n,n)}(y)\,, \qquad (x,y)\in (0,\infty)^2\,, \label{pc5} 
\end{equation}
and
\begin{equation}
f_n^{in} := f^{in} \mathbf{1}_{(0,2n)}\,. \label{pc6}
\end{equation}
It follows from \eqref{pa7} and \eqref{pc1} that $xy/(x+y)\ge 2/n$ and $K(x,y)\le 2 n^\lambda$ for $(x,y)\in (1/n,n)^2$. Thus,
\begin{equation*}
K_n(x,y) \le n^{1+\lambda} \frac{xy}{x+y}\,, \qquad (x,y)\in (0,\infty)^2\,.
\end{equation*}
Since $f_n^{in}\in X_0\cap X_1^+$, we infer from Proposition~\ref{propc1} that there is a unique strong solution $f_n\in C^1([0,\infty),X_0\cap X_1^+)$ to 
\begin{subequations}\label{pc7}
\begin{align}
\partial_t f_n(t,x) & = \mathcal{B}_n f_n(t,x) \,, \qquad (t,x)\in (0,\infty)^2\,, \label{pc7a} \\
f_n(0,x) & = f_n^{in}(x)\ge 0 \,, \qquad x\in (0,\infty)\,, \label{pc7b}
\end{align}
\end{subequations}
where $\mathcal{B}_n$ is defined by \eqref{pa1c} with $K_n$ instead of $K$. It also satisfies \eqref{pb3a} with $K_n$ instead of $K$, as well as 
\begin{equation}
M_1(f_n(t)) = M_1(f_n^{in}) \le \varrho\,, \qquad t\ge 0\,. \label{pc8}
\end{equation}
We also infer from \eqref{pa3} and Lemma~\ref{lemb3} that
\begin{equation}
\int_x^\infty y f_n(t,y)\ dy \le \int_x^\infty y f_n^{in}(y)\ dy \le \int_x^\infty y f^{in}(y)\ dy\,, \qquad t\ge 0\,. \label{pc9}
\end{equation}
In particular, the choice $x=2n$ and \eqref{pc6} imply that
\begin{equation*}
\int_{2n}^\infty y f_n(t,y)\ dy \le \int_{2n}^\infty y f_n^{in}(y)\ dy = 0\,, \qquad t\ge 0\,.
\end{equation*}
Hence, owing to the non-negativity of $f_n$,
\begin{equation}
f_n(t,x) = 0 \;\;\text{ a.e. in }\;\; (2n,\infty) \;\text{ for }\; t\ge 0\,. \label{pc10}
\end{equation}
In the same vein, moments of order higher than one are non-expansive with respect to time, provided they are initially finite.

\begin{lemma}[Higher moments]\label{lemc4}
Let $m>1$ and assume that $f^{in}\in X_m$. Then
\begin{equation*}
M_m(f_n(t)) \le M_m(f_n^{in}) \le M_m(f^{in})\,, \qquad t\ge 0\,.
\end{equation*}
\end{lemma}

\begin{proof} Let $n\ge 1$. Since $f_n^{in}$ belongs to $X_m$ with $M_m(f_n^{in})\le M_m(f^{in})$ and $f_n$ is a mass-conserving weak solution to \eqref{pc7}, Lemma~\ref{lemc4} is a straightforward consquence of Lemma~\ref{lemb4}.
\end{proof}

We next study the behaviour of $f_n$ for small sizes.

\begin{lemma}[Small size behaviour]\label{lemc5}
Recall that $\lambda=\alpha+\beta$.
\begin{itemize}
	\item [(a)] If $\lambda \in [1,2]$, then
	\begin{equation}
	M_0(f_n(t)) \le e^{\varrho(\beta_0-2) t} \left( \varrho + M_0(f^{in}) \right)\,, \qquad t\ge 0\,. \label{pc11}
	\end{equation}
	\item [(b)] If $\lambda\in [0,1)$, then
	\begin{equation}
	M_0(f_n(t)) \le \left( M_0(f^{in})^{\lambda-1} - (1-\lambda)(\beta_0-2) \varrho^\lambda t \right)^{-1/(1-\lambda)}\,, \qquad t\in \left[ 0, T_0(f^{in}) \right)\,,\label{pc12}
	\end{equation}
	recalling that $T_0(f^{in})$ is defined in \eqref{pa9}.
\end{itemize}
\end{lemma}

\begin{proof}
Let $t\ge 0$. By \eqref{pa4} and \eqref{pb3c}, 
\begin{align}
\frac{d}{dt} M_0(f_n(t)) & = \frac{1}{2} \int_0^\infty \int_0^\infty \left[ \int_0^{x+y} b(z,x,y)\ dz - 2 \right] K_n(x,y) f_n(t,x) f_n(t,y)\ dydx \nonumber \\
& \le \frac{\beta_0-2}{2} \int_0^\infty \int_0^\infty K(x,y) f_n(t,x) f_n(t,y)\ dydx \,. \label{pc13}
\end{align}

\medskip

\noindent (a) We first claim that 
\begin{equation}
K(x,y) \le x y^{\lambda-1} + x^{\lambda-1} y\,, \qquad (x,y)\in (0,\infty)^2\,. \label{pc14}
\end{equation}
Indeed, \eqref{pc14} is obvious if $\lambda=2$ as $\alpha=\beta=1$ in that case. If $\lambda\in [1,2)$, then it follows from Young's inequality that, for $(x,y)\in (0,\infty)^2$,
\begin{align*}
x^\alpha y^\beta + x^\beta y^\alpha & = (xy)^{\lambda-1} \left( x^{1-\beta} y^{1-\alpha} + x^{1-\alpha} y^{1-\beta} \right) \\
& \le (xy)^{\lambda-1} \left( x^{2-\lambda} + y^{2-\lambda} \right)\,,
\end{align*}
which completes the proof of \eqref{pc14}. 

Combining \eqref{pc13} and \eqref{pc14} gives, for $t\ge 0$,
\begin{align*}
\frac{d}{dt} M_0(f_n(t)) & \le \frac{\beta_0-2}{2} \int_0^\infty \int_0^\infty (x^{\lambda-1} y + x y^{\lambda-1}) f_n(t,x) f_n(t,y)\ dydx \\
& \le (\beta_0-2) M_1(f_n(t)) M_{\lambda-1}(f_n(t)) \,.
\end{align*}
Since $\lambda\in [1,2]$, one has $x^{\lambda-1}\le 1+x$ for $x>0$ and, using \eqref{pc8}, we end up with
\begin{equation*}
\frac{d}{dt} M_0(f_n(t)) \le \varrho (\beta_0-2) \left[ M_0(f_n(t)) + M_1(f_n(t)) \right] \le \varrho (\beta_0-2) \left[ M_0(f_n(t)) + \varrho \right] \,, \qquad t\ge 0\,.
\end{equation*}
Integrating with respect to time leads us to
\begin{align*}
M_0(f_n(t)) \le e^{\varrho(\beta_0-2) t} \left( \varrho + M_0(f_n^{in}) \right) \le e^{\varrho(\beta_0-2) t} \left( \varrho + M_0(f^{in}) \right)\,, \qquad t\ge 0\,.
\end{align*}
We have thus proved \eqref{pc11}.

\medskip

\noindent (b) Since $0\le \alpha\le \beta\le 1$ by \eqref{pa7} and \eqref{pc1}, it follows from \eqref{pc8}, \eqref{pc13}, and H\"older's inequality that
\begin{align*}
\frac{d}{dt} M_0(f_n(t)) & \le (\beta_0-2) M_\alpha(f_n(t)) M_\beta(f_n(t)) \le (\beta_0-2) M_1(f_n(t))^\lambda M_0(f_n(t))^{2-\lambda} \\
& \le (\beta_0-2) \varrho^\lambda M_0(f_n(t))^{2-\lambda}\,.
\end{align*}
Hence, after integration with respect to time,
\begin{equation*}
M_0(f_n(t)) \le \left( M_0(f_n^{in})^{\lambda-1} - (1-\lambda)(\beta_0-2) \varrho^\lambda t \right)^{-1/(1-\lambda)}\,, \qquad t\in \left[ 0, \frac{M_0(f_n^{in})^{\lambda-1}}{(1-\lambda)(\beta_0-2) \varrho^\lambda} \right)\,.
\end{equation*}
Since $\lambda\in [0,1)$ and $M_0(f_n^{in})^{\lambda-1} \ge M_0(f^{in})^{\lambda-1}$, the above inequality entails \eqref{pc12}.
\end{proof}

We next take advantage of \eqref{pc4.1} to derive a refined estimate for small sizes from Lemma~\ref{lemc5}.

\begin{lemma}[Small size behaviour revisited]\label{lemc5.1}
Let $T>0$ and $\mu_0>0$ be such that
\begin{equation}
M_0(f_n(t))\le \mu_0\,, \qquad t\in [0,T]\,. \label{pc15}
\end{equation}
There is $C_1(T)>0$ depending only on $K$, $b$, $f^{in}$, $\mu_0$, and $T$ (but not on $n\ge 1$) such that
\begin{equation*}
M_{\Phi_0}(f_n(t)) \le C_1(T)\,, \qquad t\in [0,T]\,,
\end{equation*}
the function $\Phi_0$ being defined in \eqref{pc4.1}.
\end{lemma}

\begin{proof}
Let $(x,y)\in (0,\infty)^2$. Owing to \eqref{pa3} and the non-negativity of $\Phi_0$,
\begin{equation*}
\zeta_{\Phi_0}(x,y) \le \int_0^x \Phi_0(z) \bar{b}(z,x,y)\ dz + \int_0^y \Phi_0(z) \bar{b}(z,y,x)\ dz\,. 
\end{equation*}
It now follows from \eqref{pa8},  \eqref{pc4.2}, and H\"older's inequality that
\begin{align*}
\int_0^x \Phi_0(z) \bar{b}(z,x,y)\ dz & \le x^{(p-1)/2p} \Phi_0(x) \int_0^x z^{-(p-1)/2p} \bar{b}(z,x,y)\ dz \\
& \le x^{(p-1)/2p} \Phi_0(x) \left( \int_0^x z^{-1/2}\ dz \right)^{(p-1)/p} \left( \int_0^x  \bar{b}(z,x,y)^p\ dz \right)^{1/p} \\
& \le 2^{(p-1)/p} \left( \frac{B_p}{2} \right)^{1/p} \Phi_0(x) \le B_p^{1/p} \Phi_0(x)\,.
\end{align*}
Since the same estimate is valid when $x$ and $y$ are exchanged, we conclude that
\begin{equation*}
\zeta_{\Phi_0}(x,y) \le B_p^{1/p} \left[ \Phi_0(x) + \Phi_0(y) \right]\,, \qquad (x,y)\in (0,\infty)^2\,.
\end{equation*}

Now, let $t\in [0,T]$. We infer from \eqref{pc5}, \eqref{pc7a}, the symmetry of $K$, and the above inequality that
\begin{equation*}
\frac{d}{dt} M_{\Phi_0}(f_n(t)) \le B_p^{1/p} \int_0^\infty \int_0^\infty \Phi_0(x) K(x,y) f_n(t,x) f_n(t,y)\ dydx\,.
\end{equation*}
Since 
\begin{equation}
	K(x,y) \le 2(1+x)(1+y)\,, \qquad (x,y)\in (0,\infty)^2\,, \label{zp3}
\end{equation}
by \eqref{pa7} and \eqref{pc1} and
\begin{align*}
\int_0^\infty \Phi_0(x) (1+x) f_n(t,x)\ dx & \le 2 \int_0^1 \Phi_0(x) f_n(t,x)\ dx + \Phi_0(1) \int_1^\infty (1+x) f_n(t,x)\ dx \\
& \le 2 M_{\Phi_0}(f_n(t)) + 2 \Phi_0(1) M_1(f_n(t))\,, 
\end{align*}
thanks to the monotonicity of $\Phi_0$, we further deduce from \eqref{pc8} and \eqref{pc15} that
\begin{align*}
\frac{d}{dt} M_{\Phi_0}(f_n(t)) & \le 2 B_p^{1/p} \int_0^\infty \int_0^\infty \Phi_0(x) (1+x) (1+y) f_n(t,x) f_n(t,y)\ dydx \\
& \le 2 B_p^{1/p} \left[ M_0(f_n(t)) + M_1(f_n(t)) \right] \int_0^\infty \Phi_0(x) (1+x) f_n(t,x)\ dx \\
& \le 4 B_p^{1/p} (\mu_0+\varrho) \left[ M_{\Phi_0}(f_n(t)) + \varrho \Phi_0(1) \right]\,.
\end{align*} 
Integrating with respect to time gives
\begin{equation*}
M_{\Phi_0}(f_n(t)) \le e^{4 B_p^{1/p} (\mu_0+\varrho)t} \left[ M_{\Phi_0}(f_n^{in}) + \varrho \Phi_0(1) \right] - \varrho \Phi_0(1)\,, \qquad t\in [0,T]\,,
\end{equation*}
which completes the proof due to $M_{\Phi_0}(f_n^{in}) \le M_{\Phi_0}(f^{in})<\infty$, this last inequality being a consequence of \eqref{pc4.1} and \eqref{pc6}. 
\end{proof}

We now turn to uniform integrability.

\begin{lemma}[Uniform integrability]\label{lemc6}
Let $T>0$ and $\mu_0>0$ such that \eqref{pc15} holds true. There is $C_2(T)>0$ depending only on $K$, $b$, $f^{in}$, $\mu_0$, and $T$ (but not on $n\ge 1$) such that
\begin{equation}
L_n(t) := \int_0^\infty \min\{x,1\} \Phi(f_n(t,x))\ dx \le C_2(T)\,, \qquad t\in [0,T]\,, \label{pc16}
\end{equation}
recalling that $\Phi$ is defined in \eqref{pc2}.
\end{lemma}

\begin{proof}
Let $t\in [0,T]$ and set $v(x) := \min\{x,1\}$ for $x\ge 0$. It follows from \eqref{pc7a}, the non-negativity of both $f_n$ and $\Phi'$, and Fubini's theorem that
\begin{align*}
\frac{d}{dt} L_n(t) & \le \int_0^\infty v(x) \Phi'(f_n(t,x)) \mathcal{B}_n f_n(t,x)\ dx \\
& \le \frac{1}{2} \int_0^\infty \int_0^\infty \int_0^{y+z} v(x) \Phi'(f_n(t,x)) b(x,y,z) K(y,z) f_n(t,y) f_n(t,z)\ dxdzdy\,.
\end{align*}
Since $\Phi$ is convex and satisfies $r\Phi'(r)\le 2 \Phi(r)$ for $r\ge 0$ by \cite[Proposition~7.1.9~(a)]{BLL19}, we obtain
\begin{equation*}
\Phi'(f_n(t,x)) b(x,y,z) \le \Phi(f_n(t,x)) + \Phi(b(x,y,z))\,, \qquad x\in (0,y+z)\,, \ (y,z)\in (0,\infty)^2\,, 
\end{equation*}
see \cite[Proposition~7.1.9~(b)]{BLL19}, and we deduce from the above inequalities that
\begin{align*}
\frac{d}{dt} L_n(t) & \le \frac{1}{2} \int_0^\infty \int_0^\infty \int_0^{y+z} v(x) \Phi(f_n(t,x)) K(y,z) f_n(t,y) f_n(t,z)\ dxdzdy \\
& \qquad + \frac{1}{2} \int_0^\infty \int_0^\infty \int_0^{y+z} v(x) \Phi(b(x,y,z)) K(y,z) f_n(t,y) f_n(t,z)\ dxdzdy\,.
\end{align*}
Introducing
\begin{equation*}
\kappa_p := \sup_{r\ge 0}\left\{ \frac{\Phi(r)}{r^p}\right\} \in (0,\infty)\,,
\end{equation*}
which is finite according to \eqref{pc4}, we now infer from \eqref{pa3}, \eqref{pa8}, and the inequality $v(x) x^{1-p}\le 1$, $x\in (0,\infty)$, that 
\begin{align*}
\int_0^{y+z} v(x) \Phi(b(x,y,z))\ dx & \le \kappa_p \int_0^{y+z} v(x) b(x,y,z)^p\ dx \le \kappa_p \int_0^{y+z} x^{p-1} b(x,y,z)^p\ dx \\
& \le \kappa_p \left( \int_0^y x^{p-1} \bar{b}(x,y,z)^p\ dx + \int_0^z x^{p-1} \bar{b}(x,z,y)^p\ dx \right) \\
& \le \kappa_p \left( y^{p-1} \int_0^y \bar{b}(x,y,z)^p\ dx + z^{p-1} \int_0^z \bar{b}(x,z,y)^p\ dx \right) \\
& \le \kappa_p B_p\,.
\end{align*}
Consequently,
\begin{equation}
\frac{d}{dt} L_n(t) \le \frac{L_n(t)+\kappa_p B_p}{2} \int_0^\infty \int_0^\infty K(y,z) f_n(t,y) f_n(t,z)\ dzdy\,. \label{pc17}
\end{equation}
We next infer from \eqref{pa7}, \eqref{pc8}, \eqref{pc15}, and \eqref{zp3}, that
\begin{align*}
\int_0^\infty \int_0^\infty K(y,z) f_n(t,y) f_n(t,z)\ dzdy & \le 2 \int_0^\infty \int_0^\infty (1+y) (1+z) f_n(t,y) f_n(t,z)\ dzdy \\
& \le 2 \left[ M_0(f_n(t)) + M_1(f_n(t)) \right]^2\\
& \le 2 (\mu_0 + \varrho)^2\,.
\end{align*}
Combining the above inequality with \eqref{pc17} gives
\begin{equation*}
\frac{d}{dt} L_n(t) \le (\mu_0+\varrho)^2 [L_n(t)+\kappa_p B_p]\,.
\end{equation*}
Hence, after integration with respect to time,
\begin{equation*}
L_n(t) \le [L_n(0)+\kappa_p B_p] e^{(\mu_0+\varrho)^2 t} - \kappa_p B_p\,, \qquad t\in [0,T]\,
\end{equation*}
from which \eqref{pc16} readily follows, since $L_n(0)\le \mathcal{I}$ by \eqref{pc2}, \eqref{pc6}, and the monotonicity of $\Phi$.
\end{proof}

The last step is the derivation of a time equicontinuity estimate.

\begin{lemma}[Time equicontinuity]\label{lemc7}
Let $T>0$ and $\mu_0>0$ such that \eqref{pc15} holds true. There is $C_3>0$ depending only on $K$, $b$, $f^{in}$, and $\mu_0$ (but not on $n\ge 1$) such that
\begin{equation*}
\|f_n(t_2)-f_n(t_1)\|_{X_0} \le C_3 |t_2-t_1|\,, \qquad (t_1,t_2)\in [0,T]^2\,. 
\end{equation*}
\end{lemma}

\begin{proof}
Let $t\in [0,T]$. On the one hand, it follows from Lemma~\ref{lemc2}, \eqref{pa4}, \eqref{pa7}, \eqref{pc8}, \eqref{pc15}, and \eqref{zp3} that 
\begin{align*}
& \frac{1}{2} \int_0^\infty \int_x^\infty \int_0^y b(x,y-z,z) K_n(y-z,z) f_n(t,y-z) f_n(t,z)\ dzdydx \\
& \hspace{1cm} = \frac{1}{2} \int_0^\infty \int_0^\infty \int_0^{x+y} b(z,x,y) K_n(x,y) f_n(t,x) f_n(t,y)\ dzdydx \\
& \hspace{1cm} \le \frac{\beta_0}{2} \int_0^\infty \int_0^\infty K(x,y) f_n(t,x) f_n(t,y)\ dydx \\
& \hspace{1cm} \le \beta_0 \int_0^\infty \int_0^\infty (1+x)(1+y) f_n(t,x) f_n(t,y)\ dydx \\
& \hspace{1cm} \le \beta_0 (\mu_0+\varrho)^2\,.
\end{align*}
On the other hand, by \eqref{pa7}, \eqref{pc8}, \eqref{pc15}, and \eqref{zp3},
\begin{align*}
\int_0^\infty \int_0^\infty K_n(x,y) f_n(t,x) f_n(t,y)\ dydx & \le \int_0^\infty \int_0^\infty K(x,y) f_n(t,x) f_n(t,y)\ dydx \\
& \le 2 (\mu_0+\varrho)^2\,.
\end{align*}
Gathering the previous estimates and using \eqref{pc7a} give
\begin{equation*}
\|\partial_t f_n(t)\|_{X_0} = \|\mathcal{B}_n f_n(t) \|_{X_0} \le (\beta_0+2) (\mu_0+\varrho)^2\,, \qquad t\in [0,T]\,,
\end{equation*}
from which Lemma~\ref{lemc7} readily follows.
\end{proof}

We are now in a position to prove Theorem~\ref{thma2} and begin with the case $\lambda=\alpha+\beta\in [1,2]$.

\begin{proof}[Proof of Theorem~\ref{thma2}~(a)\&(c)]
We set $v(x)=\min\{x,1\}$ for $x\ge 0$. Since $X_0\cap X_1 \subset X_v$, we infer from \eqref{pc3},  \eqref{pc8}, \eqref{pc11}, Lemma~\ref{lemc6}, and the Dunford-Pettis theorem that, for each $T>0$, there is a relatively sequentially weakly compact subset $\mathcal{K}(T)$ of $X_v$ such that $f_n(t)\in \mathcal{K}(T)$ for all $t\in [0,T]$ and $n\ge 1$. Owing to the time equicontinuity established in Lemma~\ref{lemc7}, a variant of the Arzel\`a-Ascoli theorem, see \cite[Theorem~A.3.1]{Vrab03}, ensures that $(f_n)_{n\ge 1}$ is relatively compact in $C([0,T],X_{v,w})$. Using a diagonal process, we deduce that there are $f\in C([0,\infty),X_{v,w})$ and a subsequence $(f_{n_j})_{j\ge 1}$ of $(f_n)_{n\ge 1}$ such that
\begin{equation}
\lim_{j\to\infty} \sup_{t\in [0,T]} \left| \int_0^\infty v(x) (f_{n_j}-f)(t,x) \phi(x)\ dx \right| = 0 \label{pc18}
\end{equation}
for all $\phi\in L^\infty(0,\infty)$ and $T>0$. A first consequence of \eqref{pc18} is that the validity of both \eqref{pc9} and Lemma~\ref{lemc5.1} extends to $f$. Indeed, for $T>0$, $t\in [0,T]$, $x>0$, and $R>x$, we infer from \eqref{pc9} and \eqref{pc18} that
\begin{equation*}
\int_x^R y f(t,y)\ dy = \lim_{j\to\infty} \int_x^R y f_{n_j}(t,y)\ dy \le \int_x^\infty y f^{in}(y)\ dy\,.
\end{equation*}
We then let $R\to\infty$ to conclude that
\begin{equation}
\int_x^\infty y f(t,y)\ dy \le \int_x^\infty y f^{in}(y)\ dy\,, \qquad t>0\,. \label{pc19}
\end{equation}
In the same vein, for $T>0$, $t\in [0,T]$, and $\delta\in (0,1)$, it follows from \eqref{pc11}, the monotonicity of $\Phi_0$, and Lemma~\ref{lemc5.1} that
\begin{equation*}
\int_\delta^\infty \Phi_0(x) f(t,x)\ dx = \lim_{j\to\infty} \int_\delta^\infty \Phi_0(x) f_{n_j}(t,x)\ dx \le C_1(T)\,,
\end{equation*}
and we let $\delta\to 0$ to conclude that
\begin{equation}
M_{\Phi_0}(f(t)) \le C_1(T)\,, \qquad t\in [0,T]\,. \label{pc20}
\end{equation}

We now extend \eqref{pc18} to the weak topology of $X_0\cap X_1$; that is, 
\begin{equation}
\lim_{j\to\infty} \sup_{t\in [0,T]} \left| \int_0^\infty 
(1+x) (f_{n_j}-f)(t,x) \phi(x)\ dx \right| = 0 \label{pc21}
\end{equation}
for all $\phi\in L^\infty(0,\infty)$ and $T>0$. Indeed, let $T>0$, $t\in [0,T]$, $\phi\in L^\infty(0,\infty)$, and $R>1$. Then
\begin{equation}
\begin{split}
\left| \int_0^\infty (1+x) (f_{n_j}-f)(t,x) \phi(x)\ dx \right| & \le 2 \|\phi\|_{L^\infty(0,\infty)} \int_0^{1/R}  (f_{n_j}+f)(t,x)\ dx \\
& \hspace{1cm} + \left| \int_{1/R}^R (1+x) (f_{n_j}-f)(t,x) \phi(x)\ dx \right| \\
& \hspace{1cm} + 2 \|\phi\|_{L^\infty(0,\infty)} \int_R^\infty x (f_{n_j}+f)(t,x) \ dx \,.
\end{split} \label{pc22}
\end{equation}
First, by \eqref{pc20}, the monotonicity of $\Phi_0$, and Lemma~\ref{lemc5.1},
\begin{align}
\int_0^{1/R} (f_{n_j}+f)(t,x)\ dx & \le \frac{1}{\Phi_0(1/R)} \int_0^{1/R} \Phi_0(x) (f_{n_j}+f)(t,x)\ dx \nonumber \\
& \le \frac{1}{\Phi_0(1/R)} \left[ M_{\Phi_0}(f_{n_j}(t)) + M_{\Phi_0}(f(t)) \right] \nonumber \\
& \le \frac{2}{\Phi_0(1/R)} C_1(T)\,. \label{pc23}
\end{align}
Next, thanks to \eqref{pc9} and \eqref{pc19}, 
\begin{equation}
\int_R^\infty x (f_{n_j}+f)(t,x) \ dx \le 2 \int_R^\infty x f^{in}(x)\ dx\,. \label{pc24}
\end{equation}
Gathering \eqref{pc22}, \eqref{pc23}, and \eqref{pc24} leads us to
\begin{align*}
\left| \int_0^\infty (1+x) (f_{n_j}-f)(t,x) \phi(x)\ dx \right| & \le \left[ \frac{4}{\Phi_0(1/R)} C_1(T) + 4 \int_R^\infty x f^{in}(x)\ dx \right]  \|\phi\|_{L^\infty(0,\infty)}  \\
& \hspace{1cm} + \left| \int_{1/R}^R (1+x) (f_{n_j}-f)(t,x) \phi(x)\ dx \right| \,.
\end{align*}
Since 
\begin{equation*}
\left| \frac{1+x}{v(x)} \mathbf{1}_{(1/R,R)}(x)\phi(x) \right| \le R(1+R) \|\phi\|_{L^\infty(0,\infty)}\,,
\end{equation*}
we may take the limit $j\to\infty$ in the above inequality and deduce from \eqref{pc18} that 
\begin{align*}
& \limsup_{j\to\infty} \sup_{t\in [0,T]} \left| \int_0^\infty (1+x) (f_{n_j}-f)(t,x) \phi(x)\ dx \right| \\
& \hspace{3cm} \le \left[ \frac{4}{\Phi_0(1/R)} C_1(T) + 4 \int_R^\infty x f^{in}(x)\ dx \right]  \|\phi\|_{L^\infty(0,\infty)} \,.
\end{align*}
Since $f^{in}\in X_1$, it follows from \eqref{pc4.2} that the right hand side of the above inequality vanishes in the limit $R\to\infty$, which completes the proof of \eqref{pc21}.

In particular, \eqref{pc21} implies that $f\in C([0,\infty),X_{0,w})$ and also allows us to pass to the limit in \eqref{pc8} and conclude that 
\begin{equation}
M_1(f(t)) = \lim_{j\to \infty} M_1(f_{n_j}(t)) = \lim_{j\to \infty} M_1(f_{n_j}^{in}) = M_1(f^{in})\,, \qquad t\ge 0\,. \label{pc25} 
\end{equation}
Consequently, $f$ satisfies \eqref{pa5}, the integrability property \eqref{pa5b} being a straightforward consequence of \eqref{pa5a}, \eqref{pa7}, and \eqref{pc1}.
We are left with checking \eqref{pa5c}. To this end, consider $\phi\in L^\infty(0,\infty)$ and $t>0$. On the one hand, it readily follows from \eqref{pc6} and \eqref{pc21} that
\begin{equation}
\lim_{j\to\infty} \int_0^\infty \phi(x) (f_{n_j}(t,x)-f_{n_j}^{in}(x))\ dx = \int_0^\infty \phi(x) (f(t,x)-f^{in}(x))\ dx\,. \label{pc30}
\end{equation}
On the other hand, 
\begin{equation}
\lim_{j\to\infty} \frac{\zeta_\phi(x,y) K_{n_j}(x,y)}{(1+x)(1+y)} = \frac{\zeta_\phi(x,y) K(x,y)}{(1+x)(1+y)}\,, \qquad (x,y)\in (0,\infty)^2\,, \label{pc26}
\end{equation}
while \eqref{zp2} and \eqref{zp3} entail that
\begin{equation}
\left| \frac{\zeta_\phi(x,y) K_{n_j}(x,y)}{(1+x)(1+y)} \right| \le 2 ( \beta_0 + 2) \|\phi\|_{L^\infty(0,\infty)}\,. \label{pc27}
\end{equation}
Since \eqref{pc21} implies that
\begin{equation}
\left[ (s,x,y) \mapsto f_{n_j}(s,x) f_{n_j}(s,y) \right] \rightharpoonup \left[ (s,x,y) \mapsto f(s,x) f(s,y) \right] \label{pc28}
\end{equation}
in $L^1((0,t)\times (0,\infty)^2),(1+x)(1+y) dydxds)$, we infer from \eqref{pc26}, \eqref{pc27}, \eqref{pc28}, and \cite[Proposition~2.61]{FoLe07} that
\begin{equation}
\begin{split}
& \lim_{j\to\infty} \int_0^t \int_0^\infty \int_0^\infty \zeta_\phi(x,y) K_{n_j}(x,y) f_{n_j}(s,x) f_{n_j}(s,y)\ dydxds \\
& \hspace{3cm} = \int_0^t \int_0^\infty \int_0^\infty \zeta_\phi(x,y) K(x,y) f(s,x) f(s,y)\ dydxds\,.
\end{split} \label{pc29}
\end{equation}
Since $f_{n_j}$ satisfies \eqref{pa5c} with $K_{n_j}$ instead of $K$, this property, together with \eqref{pc30} and \eqref{pc29}, ensures that $f$ satisfies \eqref{pa5c}. We have thus established that $f$ is a weak solution to \eqref{pa1} on $[0,\infty)$, which is also mass-conserving by \eqref{pc25}.

Finally, if $f^{in}\in X_m$ for some $m>1$, the assertion~(c) of Theorem~\ref{thma2} follows from Lemma~\ref{lemc4} and \eqref{pc21}, arguing as in the proof of \eqref{pc19}.
\end{proof}

\begin{proof}[Proof of Theorem~\ref{thma2}~(b)\&(c)]
The proof of Theorem~\ref{thma2}~(b)\&(c) is exactly the same as that of Theorem~\ref{thma2}~(a)\&(c), except that one has to restrict the range of $T$ and $t$ to $(0,T_0(f^{in}))$ due to Lemma~\ref{lemc5}~(b).
\end{proof}

\section{Uniqueness}\label{sec5}

The collision-induced breakage equation being a quadratic nonlinear integral equation, it is not surprising that uniqueness of weak solutions can be shown as for Smoluchowski's coagulation equation, see \cite{BLL19, EMRR05, Giri13, Stew90b} and the references therein. The proof given below is somewhat formal, as an additional truncation argument is needed to justify that the unbounded weight function $w$ defined below may indeed be used in \eqref{pa5c}. We refer to, e.g.,  \cite{Giri13, Stew90b} for the complete argument.

\begin{proof}
Set $w(x) := \max\{ x^\alpha, x\}$, $x\in (0,\infty)$ and, for $i=1,2$, let $f_i$ be a weak solution to \eqref{pa1} on $[0,T_i)$ satisfying \eqref{pa11}; that is, $M_{1+\beta}(f_i)\in L^1(0,T)$ for any $T\in (0,T_i)$. Introducing $g:= f_1+f_2$, $h:=f_1-f_2$, and $\sigma := \mathrm{sign}(h)$, we infer from Remark~\ref{rema1.5} and the symmetry of $K$ that, for $t\in (0,\min\{T_1,T_2\})$,
\begin{equation}
\frac{d}{dt} \int_0^\infty w(x) |h(t,x)|\ dx = \frac{1}{2} \int_0^\infty \int_0^\infty \zeta_{w \sigma(t)}(x,y) K(x,y) g(t,x) h(t,y)\ dydx \,. \label{pu1}
\end{equation}
We next define
\begin{equation*}
P(t,x,y) := \zeta_{w \sigma(t)}(x,y) K(x,y) \sigma(t,y)\,, \qquad (t,x,y)\in (0,\min\{T_1,T_2\})\times (0,\infty)^2\,, 
\end{equation*}
and note that
\begin{align*}
P(t,x,y) & = K(x,y) \left[ \int_0^{x+y} w(z) \sigma(t,z) \sigma(t,y) b(z,x,y)\ dz - w(x) \sigma(t,x) \sigma(t,y) - w(y) \right] \\
& \le K(x,y) \left[ \int_0^{x+y} w(z) b(z,x,y)\ dz + w(x)  - w(y) \right]\,.
\end{align*}
Since $0\le \alpha\le\beta\le 1$ and $\beta_0>2$, we infer from \eqref{pa2}, \eqref{pa4}, and \eqref{pa7} that:
\begin{itemize}
	\item if $(x,y)\in (0,1)^2$, then any $z\in (0,x+y)$ lies also in $(0,2)$, so that $w(z)\le 2^{1-\alpha} (x+y)^\alpha \le 2$ and
	\begin{equation*}
	P(t,x,y)  \le (2 \beta_ 0 + 1) \left( x^\alpha y^\beta + x^\beta y^\alpha \right) \le 6 \beta_0 y^\alpha = 6 \beta_0 w(y)\,;
	\end{equation*}
	\item if $(x,y)\in (0,1)\times (1,\infty)$, then 
	\begin{align*}
	P(t,x,y) & = K(x,y) \left( \int_0^1 z^\alpha b(z,x,y)\ dz + \int_1^{x+y} z b(z,x,y)\ dz + x^\alpha - y \right) \\
	& \le (\beta_0 + x + x^\alpha) \left( x^\alpha y^\beta + x^\beta y^\alpha \right) \le 2 (\beta_0+2) y \le 6 \beta_0 w(y)\,;
	\end{align*}
	\item if $(x,y)\in (1,\infty)\times (0,1)$, then 
	\begin{align*}
	P(t,x,y) & = K(x,y) \left( \int_0^1 z^\alpha b(z,x,y)\ dz + \int_1^{x+y} z b(z,x,y)\ dz + x - y^\alpha \right) \\
	& \le (\beta_0 + 2 x + y - y^\alpha) \left( x^\alpha y^\beta + x^\beta y^\alpha \right) \le 2 (\beta_0+2) x^{1+\beta} y^\alpha \\
	& \le 6 \beta_0 x^{1+\beta} w(y)\,;
	\end{align*}
	\item if $(x,y)\in (1,\infty)^2$, then
	\begin{align*}
	P(t,x,y) & = K(x,y) \left( \int_0^1 z^\alpha b(z,x,y)\ dz + \int_1^{x+y} z b(z,x,y)\ dz + x - y \right) \\
	& \le (\beta_0 + 2 x) \left( x^\alpha y^\beta + x^\beta y^\alpha \right) \le 2 (\beta_0+2) x^{1+\beta} y \\
	& \le 6 \beta_0 x^{1+\beta} w(y)\,.
	\end{align*}
\end{itemize}
Inserting these bounds in \eqref{pu1} and using the identity $\sigma^2 h = h = \sigma |h|$, we obtain
\begin{align*}
\frac{d}{dt} \int_0^\infty w(x) |h(t,x)|\ dx & \le \frac{1}{2} \int_0^\infty \int_0^\infty P(t,x,y) g(t,x) |h(t,y)|\ dydx \\
& \le 3 \beta_0 \int_0^\infty \int_0^\infty (1+x^{1+\beta}) w(y) g(t,x) |h(t,y)|\ dydx \\
& = 3 \beta_0 \left( M_0(g(t)) + M_{1+\beta}(g(t)) \right) \int_0^\infty w(x) |h(t,x)|\ dx\,.
\end{align*}
After integration with respect to time, we find, for $t\in (0,\min\{T_1,T_2\})$,
\begin{equation*}
\int_0^\infty w(x) |h(t,x)|\ dx \le \left( \int_0^\infty w(x) |h(0,x)|\ dx \right) \exp{\left( 3 \beta_0 \int_0^t \left( M_0(g(s)) + M_{1+\beta}(g(s)) \right)\ ds \right)}\,.
\end{equation*}
Since $M_0(g)\in L^1(0,t)$ by Definition~\ref{defa1} and $M_{1+\beta}(g)\in L^1(0,t)$ by \eqref{pa11}, the right hand side of the above inequality is finite and actually vanishes, due to $h(0)= f_1(0)-f_2(0)=0$. Consequently, $h(t)=0$ for $t\in (0,\min\{T_1,T_2\})$ and $f_1$ and $f_2$ coincide on their common time interval of existence.
\end{proof}

\section{Finite time existence}\label{sec6}

\begin{proof}[Proof of Theorem~\ref{thma3}]
Let $t\in [0,T)$ and $A>1$. We take $\phi(x) = x^\lambda \mathbf{1}_{(0,A)}(x)$, $x\in (0,\infty)$, in \eqref{pa5c}. Owing to \eqref{pa3} and \eqref{pa10},
\begin{itemize}
	\item if $(x,y)\in (0,A)^2$, then
	\begin{equation*}
	\zeta_\phi(x,y) = \int_0^x z^\lambda \bar{b}(z,x,y)\ dz + \int_0^y z^\lambda \bar{b}(z,y,x)\ dz - x^\lambda - y^\lambda \ge (\gamma_\lambda-1) (x^\lambda + y^\lambda)\,;
	\end{equation*}
	\item if $(x,y)\in (0,A)\times (A,\infty)$, then
	\begin{equation*}
	\zeta_\phi(x,y) = \int_0^x z^\lambda \bar{b}(z,x,y)\ dz + \int_0^A z^\lambda \bar{b}(z,y,x)\ dz - x^\lambda \ge (\gamma_\lambda-1) x^\lambda \ge 0\,;
	\end{equation*}
	\item if $(x,y)\in (A,\infty)\times (0,A)$, then
	\begin{equation*}
	\zeta_\phi(x,y) = \int_0^A z^\lambda \bar{b}(z,x,y)\ dz + \int_0^y z^\lambda \bar{b}(z,y,x)\ dz - y^\lambda \ge (\gamma_\lambda-1) y^\lambda\ge 0\,;
	\end{equation*}
	\item if $(x,y)\in (A,\infty)^2$, then
	\begin{equation*}
	\zeta_\phi(x,y) = \int_0^A z^\lambda \bar{b}(z,x,y)\ dz + \int_0^A z^\lambda \bar{b}(z,y,x)\ dz \ge 0\,,
	\end{equation*}
\end{itemize}	
so that, by \eqref{pa5c},
\begin{align*}
\int_0^A x^\lambda f(t,x)\ dx & \ge \int_0^A x^\lambda f^{in}(x)\ dx \\
& \hspace{1cm} + (\gamma_\lambda-1) \int_0^t \int_0^A \int_0^A (x^\lambda+y^\lambda) K(x,y) f(s,x) f(s,y)\ dydxds\,.
\end{align*}
Since $\lambda\in [0,1)$ and both $f(t)$ and $f^{in}$ belong to $X_0\cap X_1 \subset X_\lambda$, we may take the limit $A\to \infty$ in the above estimate and use Fatou's lemma to conclude that 
\begin{equation*}
M_\lambda(f(t)) \ge M_\lambda(f^{in}) + (\gamma_\lambda-1) \int_0^t \int_0^\infty \int_0^\infty (x^\lambda+y^\lambda) K(x,y) f(s,x) f(s,y)\ dydxds\,.
\end{equation*}
Thanks to Young's inequality,
\begin{equation*}
x^\lambda + y^\lambda \ge 2 (xy)^{\lambda/2}\,, \qquad K(x,y) \ge 2 (xy)^{\lambda/2}
\,, \qquad (x,y)\in (0,\infty)^2\,,
\end{equation*}
and we infer from the above two inequalities that
\begin{equation}
M_\lambda(f(t)) \ge \Lambda(t) := M_\lambda(f^{in}) + 4 (\gamma_\lambda-1) \int_0^t M_\lambda(f(s))^2\ ds\,, \qquad t\in [0,T)\,. \label{pc40}
\end{equation}
Hence, 
\begin{equation*}
\frac{d\Lambda}{dt}(t) \ge 4 (\gamma_\lambda-1) \Lambda(t)^2\,, \qquad t\in [0,T)\,,
\end{equation*}
from which we deduce that 
\begin{equation*}
\frac{1}{M_\lambda(f^{in})} = \frac{1}{\Lambda(0)} \ge \frac{1}{\Lambda(0)} - \frac{1}{\Lambda(t)} \ge 4 (\gamma_\lambda-1)t\,, \qquad t\in [0,T)\,.
\end{equation*}
Letting $t\to T$ implies that $T\le 1/[4(\gamma_\lambda-1) M_\lambda(f^{in})] < \infty$.
\end{proof}

\section{Non-existence}\label{sec7}

Let $f$ be a mass-conserving weak solution to \eqref{pa1} on $[0,T_*)$ for some $T_*\in (0,\infty]$, with initial condition $f^{in}\in X_0\cap X_1^+$ satisfying $f^{in}\in X_{m_0}$ for some $m_0>1$ and $\varrho := M_1(f^{in})>0$. Then
\begin{equation}
M_1(f(t)) = \varrho = M_1(f^{in})\,, \qquad t\in [0,T_*)\,. \label{pe18}
\end{equation}
Also, since $f^{in}\in X_{m_0}$ with $m_0>1$ and $b$ satisfies \eqref{pa3}, Lemma~\ref{lemb4} guarantees that 
\begin{equation}
M_{m_0}(f(t)) \le M_{m_0}(f^{in})\,, \qquad t\in [0,T_*)\,. \label{pe22}
\end{equation}

The first step of the proof of Theorem~\ref{thma6} is strongly inspired by the proof of the occurrence of instantaneous gelation  for Smoluchowski's coagulation equation, see \cite{CadC92, vanD87c} and \cite[Theorem~9.2.2]{BLL19}. We actually establish that the mass-conserving feature \eqref{pe18} of $f$ implies that all sublinear moments are finite for all $t\in [0,T_*)$, including moments of negative order. Thus, the very existence of such a solution requires at least that $M_m(f^{in})<\infty$ for all $m<0$; that is, $f^{in}$ shall vanish rapidly for small sizes.

\begin{lemma}\label{leme1}
Let $m\in (-\infty,1)$. For any $T\in (0,T_*)$,
\begin{equation*}
\sup_{t\in [0,T]} M_m(f(t)) < \infty\,.
\end{equation*}
In particular, $f^{in}\in X_m$.
\end{lemma}

\begin{proof}
Fix $T_1\in (T,T_*)$. For $l\in \mathbb{N}$, $l\ge 1$, and $t\in [0,T_1]$, we define 
\begin{equation}
J_l(t) := \int_{1/l}^\infty x f(t,x)\ dx \,, \qquad I_l(t) := \int_0^{1/l} x f(t,x)\ dx\,, \label{pe19}
\end{equation}
and note that \eqref{pe18} entails that
\begin{equation}
I_l(t) + J_l(t) = \varrho =M_1(f^{in})\,, \qquad t\in [0,T_1]\,. \label{pe20}
\end{equation}
Since 
\begin{equation*}
J_l(t) \le J_{l+1}(t)\,, \quad l\ge 1\,, \;\;\text{ and }\;\; \lim_{l\to \infty} J_l(t) = \varrho \;\text{ for }\; t\in [0,T_1]\,,
\end{equation*}
Dini's theorem implies that $(J_l)_{l\ge 1}$ converges uniformly to the constant function $\varrho$ in $[0,T_1]$. Consequently, there is $l_0\ge 1$ such that
\begin{equation}
\frac{\varrho}{2} \le J_l(t)\,, \qquad t\in [0,T_1]\,, \quad l\ge l_0\,. \label{pe21}
\end{equation}

Next, let $l\in\mathbb{N}$, $l\ge l_0$. As $\phi_l: x\mapsto x \mathbf{1}_{(0,1/l)}(x)$ belongs to $L^\infty(0,\infty)$, we infer from Remark~\ref{rema1.5} and \eqref{pe20} that both $I_l$ and $J_l$ belong to $W^{1,1}(0,T_1)$ and satisfy, for almost every $t\in (0,T_1)$,
\begin{equation*}
- \frac{d}{dt} J_l(t) = \frac{d}{dt} I_l(t) = \frac{1}{2} \int_0^\infty \int_0^\infty \zeta_{\phi_l}(x,y) K(x,y) f(t,x) f(t,y)\ dydx \,.
\end{equation*}
Owing to \eqref{pa3},
\begin{itemize}
	\item if $(x,y)\in (0,1/l)^2$, then $\zeta_{\phi_l}(x,y) = 0$;
	\item if $(x,y)\in (0,1/l)\times (1/l,\infty)$, then
	\begin{equation*}
	\zeta_{\phi_l}(x,y) = \int_0^{1/l} z \bar{b}(z,y,x)\ dz = y - \int_{1/l}^y z \bar{b}(z,y,x)\ dz \ge 0\,;
	\end{equation*}
	\item if $(x,y)\in (1/l,\infty)\times (0,1/l)$, then
	\begin{equation*}
	\zeta_{\phi_l}(x,y) = \int_0^{1/l} z \bar{b}(z,x,y)\ dz = x - \int_{1/l}^x z \bar{b}(z,x,y)\ dz \ge 0\,;
	\end{equation*}
	\item if $(x,y)\in (1/l,\infty)^2$, then
	\begin{align*}
	\zeta_{\phi_l}(x,y) & = \int_0^{1/l} z \bar{b}(z,x,y)\ dz + \int_0^{1/l} z \bar{b}(z,y,x)\ dz \\
	& = x - \int_{1/l}^x z \bar{b}(z,x,y)\ dz + y - \int_{1/l}^y z \bar{b}(z,y,x)\ dz \ge 0\,.
	\end{align*}
\end{itemize}
Consequently, by \eqref{pa7} and \eqref{pa12},
\begin{align*}
- \frac{d}{dt} J_l(t) & \ge \frac{1}{2} \int_0^{1/l} \int_{1/l}^\infty \left[ y - \int_{1/l}^y z \bar{b}(z,y,x)\ dz \right] K(x,y) f(t,x) f(t,y)\ dydx \\
& = \frac{1}{2 l^{\nu+2}} \int_0^{1/l} \int_{1/l}^\infty y^{-\nu-1} K(x,y) f(t,x) f(t,y)\ dydx \\
& \ge \frac{1}{2 l^{\nu+2}} \int_0^{1/l} \int_{1/l}^\infty x^\beta y^{\alpha-\nu-1} f(t,x) f(t,y)\ dydx \\
& \ge \frac{1}{2 l^{\nu+2}} \int_0^{1/l} \int_{1/l}^\infty x y^{\lambda -\nu-2} f(t,x) f(t,y)\ dydx \,.
\end{align*}
Hence, 
\begin{equation}
\frac{d}{dt} J_l(t) \le - \frac{I_l(t)}{2 l^{\nu+2}}  \int_{1/l}^\infty y^{\lambda -\nu-2}  f(t,y)\ dy \,. \label{pe23}
\end{equation}
Since $\lambda-\nu-2<0$ and $l\ge l_0$, it follows from H\"older's inequality, \eqref{pe22}, and \eqref{pe21} that
\begin{align*}
\frac{\varrho}{2} & \le J_l(t) = \int_{1/l}^\infty y f(t,y)\ dy \\
& \le \left( \int_{1/l}^\infty y^{\lambda-\nu-2} f(t,y)\ dy \right)^{(m_0-1)/(m_0+\nu+2-\lambda)} \left( \int_{1/l}^\infty y^{m_0} f(t,y)\ dy \right)^{(3+\nu-\lambda)/(m_0+\nu+2-\lambda)} \\
& \le M_{m_0}(f^{in})^{(3+\nu-\lambda)/(m_0+\nu+2-\lambda)} \left( \int_{1/l}^\infty y^{\lambda-\nu-2} f(t,y)\ dy \right)^{(m_0-1)/(m_0+\nu+2-\lambda)}\,.
\end{align*}
Hence,
\begin{equation}
\int_{1/l}^\infty y^{\lambda-\nu-2} f(t,y)\ dy \ge 2\delta_1 := \left( \frac{\varrho}{2} \right)^{(m_0+\nu+2-\lambda)/(m_0-1)} M_{m_0}(f^{in})^{(\lambda-\nu-3)/(m_0-1)}>0\,. \label{pe24}
\end{equation}
Combining \eqref{pe23} and \eqref{pe24}, we end up with
\begin{equation*}
\frac{d}{dt} J_l(t) \le - \frac{\delta_1}{l^{\nu+2}} I_l(t)\,, \qquad t\in (0,T_1)\,.
\end{equation*}
After integration with respect to time, we deduce from \eqref{pe20} that, for $0\le t \le \tau \le T_1$,
\begin{equation*}
I_l(t) - I_l(\tau) = J_l(\tau) - J_l(t) \le - \frac{\delta_1}{l^{\nu+2}} \int_t^\tau I_l(s)\ ds\,.
\end{equation*}
Hence,
\begin{equation}
I_l(t) + \frac{\delta_1}{l^{\nu+2}} \int_t^\tau I_l(s)\ ds \le I_l(\tau)\,, \qquad 0\le t \le \tau \le T_1\,. \label{pe25}
\end{equation}
Now, for $\tau\in [t,T_1]$, we set
\begin{equation*}
Q_l(\tau) := I_l(t) + \frac{\delta_1}{l^{\nu+2}} \int_t^\tau I_l(s)\ ds
\end{equation*}
and deduce from \eqref{pe25} that $Q_l$ satisfies the differential inequality
\begin{equation*}
\frac{dQ_l}{d\tau}(\tau) = \frac{\delta_1}{l^{\nu+2}} I_l(\tau) \ge \frac{\delta_1}{l^{\nu+2}} Q_l(\tau)\,, \qquad \tau\in (t,T_1)\,.
\end{equation*}
Equivalently,
\begin{equation*}
\frac{d}{d\tau} \left( e^{-\delta_1 \tau l^{-\nu-2}} Q_l(\tau)\right) \ge 0\,, \qquad t\in (t,T_1)\,,
\end{equation*}
so that, after integration with respect to $\tau$ over $(t,T_1)$,
\begin{equation*}
e^{-\delta_1 T_1 l^{-\nu-2}} Q_l(T_1) \ge e^{-\delta_1 tl^{-\nu-2}} Q_l(t) = e^{-\delta_1 tl^{-\nu-2}} I_l(t)\,.
\end{equation*}
Since $Q_l(T_1)\le I_l(T_1)\le \varrho$ by \eqref{pe20} and \eqref{pe25}, we conclude that
\begin{equation}
I_l(t) \le \varrho e^{-\delta_1 (T_1-t)l^{-\nu-2}}\,, \qquad t\in [0,T_1]\,. \label{pe26}
\end{equation}

Consider now $m<1$ and $L\ge 2$. For $t\in [0,T]$, we infer from \eqref{pe20} and \eqref{pe26} that
\begin{align*}
\int_{1/L}^\infty x^m f(t,x)\ dx & = \int_1^\infty x^m f(t,x)\ dx + \sum_{l=1}^{L-1} \int_{1/(l+1)}^{1/l} x^m f(t,x)\ dx \\
& \le \int_1^\infty x f(t,x)\ dx + \sum_{l=1}^{L-1} (l+1)^{1-m} \int_{1/(l+1)}^{1/l} x f(t,x)\ dx \\
& \le \varrho + \sum_{l=1}^{L-1} (l+1)^{1-m} I_l(t) \\
& \le \varrho + \varrho \sum_{l=1}^{L-1} (l+1)^{1-m}  e^{-\delta_1 (T_1-t)l^{-\nu-2}} \\
& \le \varrho + \varrho \sum_{l=1}^{L-1} (l+1)^{1-m}  e^{-\delta_1 (T_1-T)l^{-\nu-2}}\,.
\end{align*}
Since $\nu+2>0$ and $T_1>T$, the series on the right hand side of the above inequality is convergent and we may take the limit $L\to\infty$ to conclude that
\begin{equation*}
M_m(f(t)) \le \varrho + \varrho \sum_{l=1}^\infty (l+1)^{1-m}  e^{-\delta_1 (T_1-T)l^{-\nu-2}}\,, \qquad t\in [0,T]\,,
\end{equation*}
and complete the proof of Lemma~\ref{leme1}.
\end{proof}

We are now in a position to prove Theorem~\ref{thma6} and here deviate from the analysis performed in \cite{BLL19, CadC92, vanD87c}. Indeed, due to the limited integrability properties of the daughter distribution function $b$, we cannot compute the time evolution of moments with arbitrary negative order but it turns out that evaluating moments of order $m\in (-\nu-1,0)$ provides the needed information. 

\begin{proof}[Proof of Theorem~\ref{thma6}]
Let $m\in (-\nu-1,0)$. For $\varepsilon\in (0,1)$ and $x\in (0,\infty)$, we set $\phi_{m,\varepsilon}(x) := (x+\varepsilon)^m$ and note that the non-positivity of $\nu$ entails that
\begin{align*}
\zeta_{\phi_{m,\varepsilon}}(x,y) & = \frac{\nu+2}{x^{\nu+1}} \int_0^x z^\nu (z+\varepsilon)^m\ dz + \frac{\nu+2}{y^{\nu+1}} \int_0^y z^\nu (z+\varepsilon)^m\ dz - (x+\varepsilon)^m - (y+\varepsilon)^m \\
& \ge \frac{\nu+2}{x^{\nu+1}} \int_0^x (z+\varepsilon)^{m+\nu}\ dz + \frac{\nu+2}{y^{\nu+1}} \int_0^y (z+\varepsilon)^{m+\nu}\ dz - (x+\varepsilon)^m - (y+\varepsilon)^m \\
& = \frac{\nu+2}{m+\nu+1} \frac{(x+\varepsilon)^{m+\nu+1}}{x^{\nu+1}} - (x+\varepsilon)^m +  \frac{\nu+2}{m+\nu+1} \frac{(y+\varepsilon)^{m+\nu+1}}{y^{\nu+1}} - (y+\varepsilon)^m \\
& \ge \frac{1-m}{m+\nu+1} \frac{(x+\varepsilon)^{m+\nu+1}}{x^{\nu+1}} +  \frac{1-m}{m+\nu+1} \frac{(y+\varepsilon)^{m+\nu+1}}{y^{\nu+1}} \ge 0\,.
\end{align*}
Since $\phi_{m,\varepsilon}\in L^\infty(0,\infty)$, we infer from \eqref{pa5c}, the symmetry of $K$, and the above inequality that, for $t\in (0,T_*)$,
\begin{align*}
& \int_0^\infty (x+\varepsilon)^m f(t,x)\ dx \ge \int_0^\infty (x+\varepsilon)^m f^{in}(x)\ dx \\
& \hspace{3cm} +  \frac{1-m}{m+\nu+1} \int_0^t \int_0^\infty \int_0^\infty \frac{(x+\varepsilon)^{m+\nu+1}}{x^{\nu+1}} K(x,y) f(s,x) f(s,y)\ dydxds\,.
\end{align*}
Since both $f(t)$ and $f^{in}$ belong to $X_m$ by Lemma~\ref{leme1}, we may let $\varepsilon\to 0$ in the above inequality and use Fatou's lemma and the specific choice \eqref{pa7} of $K$ to obtain that
\begin{equation}
M_m(f(t)) \ge M_m(f^{in}) + \frac{1-m}{m+\nu+1} \int_0^t M_{m+\alpha}(f(s)) M_\beta(f(s))\ ds\,, \qquad t\in [0,T_*)\,. \label{pe27}
\end{equation}

Next, since $\beta \le 1 < m_0$ and $m+\alpha<m<1$, it follows from \eqref{pe18}, \eqref{pe22}, and H\"older's inequality that, for $s\in [0,T_*)$,
\begin{align*}
\varrho & = M_1(f(s)) \le M_{m_0}(f(s))^{(1-\beta)/(m_0-\beta)} M_\beta(f(s))^{(m_0-1)/(m_0-\beta)} \\
& \le M_{m_0}(f^{in})^{(1-\beta)/(m_0-\beta)} M_\beta(f(s))^{(m_0-1)/(m_0-\beta)}
\end{align*} 
and
\begin{align*}
M_m(f(s)) & \le M_1(f(s))^{-\alpha/(1-\alpha-m)} M_{m+\alpha}(f(s))^{(1-m)/(1-\alpha-m)} \\
& = \varrho^{-\alpha/(1-\alpha-m)} M_{m+\alpha}(f(s))^{(1-m)/(1-\alpha-m)}\,.
\end{align*}
Consequently,
\begin{equation}
M_{m+\alpha}(f(s)) M_\beta(f(s)) \ge \delta_2 \varrho^{\alpha/(1-m)} M_m(f(s))^{(1-\alpha-m)/(1-m)}\,, \qquad s\in [0,T_*)\,, \label{pe28}
\end{equation}
with
\begin{equation*}
\delta_2 := \varrho^{(m_0-\beta)/(m_0-1)} M_{m_0}(f^{in})^{(\beta-1)/(m_0-1)}>0\,.
\end{equation*}

Combining \eqref{pe27} and \eqref{pe28} leads us to the integral inequality 
\begin{equation}
M_m(f(t)) \ge N_m(t) := M_m(f^{in}) + \frac{(1-m)\delta_2}{m+\nu+1} \varrho^{\alpha/(1-m)} \int_0^t M_m(f(s))^{(1-\alpha-m)/(1-m)}\ ds \label{pe29}
\end{equation}
for $t\in [0,T_*)$. It readily follows from \eqref{pe29} that $N_m$ satisfies the following differential inequality
\begin{equation*}
\frac{dN_m}{dt}(t) \ge \frac{(1-m)\delta_2}{m+\nu+1} \varrho^{\alpha/(1-m)} N_m(t)^{(1-\alpha-m)/(1-m)}\,, \qquad t\in [0,T_*)\,.
\end{equation*}
After integration, we obtain 
\begin{equation*}
N_m(t)^{\alpha/(1-m)} \le N_m(0)^{\alpha/(1-m)} + \frac{\alpha \delta_2}{m+\nu+1} \varrho^{\alpha/(1-m)} t\,, \qquad t\in [0,T_*)\,.
\end{equation*}
Since $\alpha<0$, we deduce from \eqref{pe29} and the above inequality that
\begin{equation*}
M_m(f(t))^{\alpha/(1-m)} \le N_m(t)^{\alpha/(1-m)} \le M_m(f^{in})^{\alpha/(1-m)} + \frac{\alpha \delta_2}{m+\nu+1} \varrho^{\alpha/(1-m)} t\,, \qquad t\in [0,T_*)\,.
\end{equation*}
Owing to the non-negativity of $M_m(f(t))^{\alpha/(1-m)} $, we further obtain
\begin{equation*}
	t \le \frac{m+\nu+1}{|\alpha| \delta_2} \varrho^{-\alpha/(1-m)} M_m(f^{in})^{\alpha/(1-m)}\,, \qquad t\in [0,T_*)\,.
\end{equation*}
Letting $t\to T_*$ in the previous inequality implies that
\begin{equation*}
T_* \le \frac{m+\nu+1}{|\alpha| \delta_2} \varrho^{-\alpha/(1-m)} M_m(f^{in})^{\alpha/(1-m)}\,.
\end{equation*}
Since $f^{in}\in X_{-\nu-1}$ by Lemma~\ref{leme1}, the right hand side of the above inequality converges to zero as $m\to -\nu-1$, from which we deduce that $T_*=0$, thereby completing the non-existence proof.
\end{proof}

\section*{Acknowledgements}

 Part of this work was done while PhL enjoyed the hospitality of the Department of Mathematics, Indian Institute of Technology Roorkee. We thank the referees for their careful reading of and their remarks on the manuscript.
 
 \appendix
 \section{Improved integrability for small sizes}\label{sec.a}
 
 We recall in this section an improved integrability property of integrable functions near zero, which is established in \cite{GiLa} and can be viewed as a variant of the de la Vall\'ee-Poussin theorem \cite{dlVP15}.
 
 \begin{lemma}\label{leap1}
 	Consider $h\in X_0$ and $\theta\in (0,1)$. There is a non-negative convex and non-increasing function $\Phi_0\in C^1((0,\infty))$ depending on $h$ and $\theta$ such that
 	\begin{equation}
 	\int_0^\infty \Phi_0(x) |h(x)|\ dx < \infty \,, \label{pap1}
 	\end{equation}
 	and
 	\begin{equation}
 	\lim_{x\to 0} \Phi_0(x) = \infty\,, \qquad \lim_{x\to 0} x^\theta \Phi_0(x) = 0\,, \qquad x \mapsto x^\theta \Phi_0(x) \;\text{ is non-decreasing}\,. \label{pap2}
 	\end{equation}
 \end{lemma}
 
\bibliographystyle{siam}
\bibliography{CIBE}

\end{document}